\documentclass[11pt,reqno]{amsart}

\usepackage{amssymb,latexsym, tikz-cd}
\usepackage{url} 
\usepackage{hyperref}
\usepackage{color}
\usepackage[utf8]{inputenc}

\calclayout

\makeatletter
\newcommand{\monthyear}[1]{%
  \def\@monthyear{\uppercase{#1}}}
\newcommand{\volnumber}[1]{%
  \def\@volnumber{\uppercase{#1}}}
\AtBeginDocument{%
\def\ps@plain{\ps@empty
  \def\@oddfoot{\@monthyear \hfil \thepage}%
  \def\@evenfoot{\thepage \hfil \@volnumber}}
\def\ps@firstpage{\ps@plain}
\def\ps@headings{\ps@empty
  \def\@evenhead{%
    \setTrue{runhead}%
    \def\thanks{\protect\thanks@warning}%
    \uppercase{The Fibonacci Quarterly}\hfil}%
  \def\@oddhead{%
    \setTrue{runhead}%
    \def\thanks{\protect\thanks@warning}%
    \hfill\uppercase{Hypergeometric Template}}%
  \let\@mkboth\markboth
  \def\@evenfoot{%
    \thepage \hfil \@volnumber}%
  \def\@oddfoot{%
    \@monthyear \hfil \thepage}%
  }%
\footskip=25pt
\pagestyle{headings}%
}
\makeatother

\theoremstyle{plain}
\numberwithin{equation}{section}
\newtheorem{thm}{Theorem}[section]

\newtheorem{lem}[thm]{Lemma}

\newtheorem{defi}[thm]{Definition}

\newtheorem{pro}[thm]{Property}

\begin{document}
\monthyear{Month Year}
\volnumber{Volume, Number}
\setcounter{page}{1}

\title{Winning Strategies for Generalized Zeckendorf Game}
\author{Steven~J.~Miller}
\address{Department of Mathematics and Statistics, Williams College, Williamstown, MA 01267}
\email{\textcolor{blue}{\href{mailto:sjm1@williams.edu}{sjm1@williams.edu}},  \textcolor{blue}{\href{mailto:Steven.Miller.MC.96@aya.yale.edu}{Steven.Miller.MC.96@aya.yale.edu}}}
\thanks{This work was completed as part of the 2022 Polymath Jr program. We thank our colleagues there, and the referee, for helpful comments on this research.}
\author{Eliel Sosis}
\address{University of Michigan, Ann Arbor, MI 48109}
\email{\textcolor{blue}{\href{mailto:esosis@umich.edu}{esosis@umich.edu}}}
\author{Jingkai~Ye}
\address{New York University, New York, NY 10012}
\email{\textcolor{blue}{\href{mailto:jy3695@nyu.edu}{jy3695@nyu.edu}}}

\begin{abstract}
Zeckendorf proved that every positive integer $n$ can be written uniquely as the sum of non-adjacent Fibonacci numbers; a similar result holds for other positive linear recurrence sequences. These legal decompositions can be used to construct a game that starts with a fixed integer $n$, and players take turns using moves relating to a given recurrence relation. The game eventually terminates in a unique legal decomposition, and the player who makes the final move wins. 

For the Fibonacci game, Player $2$ has the winning strategy for all $n>2$. We give a non-constructive proof that for the two-player $(c, k)$-nacci game, for all $k$ and sufficiently large $n$, Player $1$ has a winning strategy when $c$ is even and Player $2$ has a winning strategy when $c$ is odd. Interestingly, the player with the winning strategy can make a mistake as early as the $c + 1$ turn, in which case the other player gains the winning strategy. Furthermore, we proved that for the $(c, k)$-nacci game with players $p \ge c + 2$, no player has a winning strategy for any $n \ge 3c^2 + 6c + 3$. We find a stricter lower boundary, $n \ge 7$, in the case of the three-player $(1, 2)$-nacci game. Then we extend the result from the multiplayer game to multialliance games, showing which alliance has a winning strategy or when no winning strategy exists for some special cases of multialliance games.
\end{abstract}

\maketitle

\section{Introduction}

\subsection{History}
The Fibonacci numbers, defined by $F_1 = 1, F_2 = 2$ and $F_{n+1} = F_n + F_{n-1}$, are a fascinating sequence with many interesting properties and applications \cite{Kos}. Zeckendorf \cite{Ze} proved that every positive integer $n$ can be uniquely written as the sum of distinct, non-adjacent Fibonacci numbers. This sum is called the Zeckendorf decomposition of $n$, and is why we defined the sequence to start 1, 2, 3, as if we start with a 0 or with two 1's we lose uniqueness. There is an extensive literature on Zeckendorf expansions and their generalizations to other recurrence relations; see \cite{Al, BEFM1, BEFM2, BBGILMT, BILMT, Br, CHHMPV, Day, DDKMMV, FGNPT, GT, GTNP, Ha, HW, Ho, Ke, KKMW, Len, LLMMSXZ, MMMS, MMMMS, MW1, MW2, Ste} and the references therein.

Baird-Smith, Epstein, Flint, and Miller \cite{BEFM1, BEFM2} created the Zeckendorf Game, which is played on decompositions of integers as sums of Fibonacci numbers; we describe the game in the next subsection. They proved that every game terminates in the Zeckendorf decomposition, and in the two-player version, for any integer $n > 2$, Player $2$ always has a winning strategy. 

We explore a generalization of their game. Similar to many other problems in the field, we find significant progress is possible if we restrict what recurrences we study. This is similar to work of Brower et. al. \cite{BILMT}, who proved that if the recurrence relation has all coefficients positive integers, then the analysis of the distribution of gaps between summands can be computed cleanly (the probability of a gap of length $k \ge 2$ decays geometrically), and to Cordwell et. al. \cite{CHHMPV}, who proved that if the sequence of coefficients of the recurrence are positive non-increasing integers, then the generalized Zeckendorf decomposition is summand minimal.




\subsection{The Zeckendorf Game}
We quote from \cite{BEFM2} to describe the rules and play of the Zeckedorf game created by Baird-Smith, Epstein, Flint and Miller.

We first introduce some notation. When we write $\{1^n\}$ or $\{F_{1}^n\}$, we mean $n$ copies of $1$, the first Fibonacci number. If we have $7$ copies of $F_{1}$, $4$ copies of $F_{3}$, and $2$ copies of $F_{4}$, we write either $\{F_{1}^7 \wedge F_{3}^4 \wedge F_{4}^2\}$ or $\{1^7 \wedge 3^4 \wedge 5^2\}$.

\begin{defi}\label{2pZeck}
(The Two-Player Zeckendorf Game). At the beginning of the game, there is an unordered list of $n$ $1$’s, so the
initial list is $\{F_{1}^n\}$. On each turn, a player can do one of the following moves. 
\begin{enumerate}
   \item If the list contains two consecutive Fibonacci numbers, $F_{i-1}$ and $F_{i}$, then a player can change these to $F_{i+1}$. We denote this move $\{F_{i-1} \wedge F_{i} \rightarrow F_{i+1}\}$.
   \item If the list has two of the same Fibonacci number, $F_{i}$ and $F_{i}$, then
   \begin{enumerate}
     \item if $i = 1$, a player can change $F_{1}$ and $F_{1}$ to $F_{2}$, denoted by $\{F_{1} \wedge F_{1} \rightarrow F_{2}\}$,
     \item if $i = 2$, a player can change $F_{2}$ and $F_{2}$ to $F_{1}$ and $F_{3}$, denoted by $\{F_{2} \wedge F_{2} \rightarrow F_{1} \rightarrow F_{3}\}$, and
     \item if $i \geq 3$, a player can change $F_{i}$ and $F_{i}$ to $F_{i-2}$ and $F_{i+1}$, denoted by $\{F_{i} \wedge F_{i} \rightarrow F_{i-2} \wedge F_{i+1}\}$.
   \end{enumerate}
\end{enumerate}
The players alternative moving. The game ends when a player makes the Zeckendorf decomposition of $n$, for which no further moves are possible.
\end{defi}

Proofs that the Zeckendorf game is playable and ends at the Zeckendorf decomposition can be found in \cite{BEFM2}. The same paper also gives a non-constructive proof that for all $n > 2$, player $2$ has a winning strategy. There are many papers expanding these results in several directions, including the winning strategy of multiplayer and multialliance Zeckendorf Games \cite{CDH1}, and bounds of Zeckendorf games \cite{CDH2}. Furthermore, there are some other interesting games stemming from the Fibonacci Zeckendorf Game, including the Fibonacci Quilt game \cite{MN}, Bergman game \cite{BDD}, Deterministic Zeckendorf game \cite{LLMMSXZ} and Generalized Zeckendorf games \cite{BCD}. As this paper mainly focuses on Generalized Zeckendorf games, below is a detailed introduction of this game.

\subsection{Generalized Zeckendorf Games}
The Zeckendorf game as described so far only concerns a game on the Fibonacci sequence. However, using the Generalized Zeckendorf theorem, \cite{BEFM1} defined the Generalized Zeckendorf Game for certain positive linear recurrence sequences. We quote from \cite{BEFM1} to describe the sequences and how to play the game. 

\begin{defi}\label{ck-nacci}((c,k)-nacci Numbers).
We call any sequence defined by a recurrence $S_{i+1} = c S_{i} + c S_{i-1} + \cdots + c S_{i-k}$ a generalized $k$-nacci sequence with constant $c$. The initial conditions are as follows: $S_{1} = 1$, and for $1 \leq n < k + 1$ we have $S_{i+1} = c S_{i} + c S_{i-1} + \cdots + c S_{1} + 1$. The terms $S_{t}$ are called (c, k)-nacci numbers.
\end{defi}

\begin{defi}\label{2pGenZeck}(The Two-Player $(c,k)$-nacci Zeckendorf Game).
Two people play the $(c, k)$-nacci Zeckendorf Game (a special Generalized Zeckendorf Game), for the $(c, k)$-nacci numbers. At the beginning of the game, we have an unordered list of $n$ $1$’s. If $i < k + 1$, $S_{i+1} = c S_{i} + c S_{i-1} + \cdots + c S_{1} + 1$. If $i \geq k$, $S_{i+1} = c S_{i} + c S_{i-1} + \cdots + c S_{i-k}$. Therefore our initial list is $\{S^n_{1}\}$. On each turn we can do one of the following moves.
\begin{enumerate}
   \item If our list contains $k + 1$ consecutive $k$-nacci numbers each with multiplicity $c$, then we can change these to $S_{i+1}$. We denote this move $\{c S_{i-k} \wedge c S_{i-k + 1} \wedge \cdots \wedge c S_{i} \rightarrow S_{i+1}\}$.
   \item If our list contains consecutive $k$-nacci numbers with multiplicity $c$ up to an index less than or equal to $k$, and $S_{1}$ with multiplicity $c + 1$, we can do the move $\{(c + 1) S_{1} \wedge c S_{2} \wedge \cdots \wedge c S_{i} \rightarrow S_{i+1}\}$.
   \item If the list has $c + 1$ of the same $k$-nacci number $S_{i}$, then
   \begin{enumerate}
     \item If $i = 1$, then we can change $(c + 1) S_{1}$ to $S_{2}$, denoting this move $\{(c + 1) S_{1} \rightarrow S_{2}\}$;
     \item If $1 < i < k + 1$, then we can change $(c + 1) S_{i}$ to $S_{i+1}$, denoted by $\{(c+1) S_{i} \rightarrow S_{i+1}\}$;
     \item If $i = k + 1$, then we can do the move $\{(c + 1) S_{i} \rightarrow S_{i+1} \wedge S_{1}\}$; and
     \item If $i > k + 1$, then we can do the move $\{(c + 1) S_{i} \rightarrow S_{i+1} \wedge c S_{i-k-1}\}$.
   \end{enumerate}
\end{enumerate}
Players alternate moving until no moves remain.
\end{defi}

Proofs that certain Generalized Zeckendorf games are playable and end at Generalized Zeckendorf decomposition can be found in \cite{BEFM1}. For other results on the Zeckendorf game and some of its generalizations, see \cite{BCD, BDD, CDH1, CDH2, LLMMSXZ, MN}.  


\subsection{Main Results}
We find winning strategies for some Generalized Zeckendorf games, and also consider multiplayer and multialliance versions of the game. First, we start with exploring the Tribonacci Game, which is a special case of the Generalized $(c,k)$-nacci Zeckendorf Game. In the following two results, we find which player has a winning strategy for the two-player and multiplayer Tribonacci Games. 
\begin{lem}\label{2pTrib}
For all $n > 9$, Player $2$ has the winning strategy for the two-player Tribonacci Game.
\end{lem}
\begin{thm}\label{thm1} 
For the multiplayer Tribonacci game, when $n \geq 7$, for any $p \geq 3$ no player has a winning strategy.
\end{thm}
Next, we extend our results to the Generalized $(c,k)$-nacci Zeckendorf Game. Our main result, shown below, is for the two-player $(c,k)$-nacci game: we find a general pattern for which player has a winning strategy for any positive integers $c$ and $k$. 
\begin{thm} \label{thm3}
In a $(c,k)$-nacci game, for any $c\geq 1, k \geq 1, n\geq (c+1)^3+(c+1)$, when $c$ is odd, player $2$ always has a winning strategy; when $c$ is even, player $1$ always has a winning strategy.
\end{thm}
This significantly extends the previous results of winning strategies for the $2$-player Zeckendorf game. One interesting question to ask is how quickly a player with a winning strategy can make a mistake to lose the winning strategy, which we address in the following theorem.
\begin{thm}\label{thm4}
For the $(c,k)$-nacci game, where $k > 1$, if the player with the winning strategy makes a mistake as early as turn $c + 1$, the opposing player can steal the winning strategy.
\end{thm}
Finally, we shift our focus to the multiplayer and multialliance Generalized Zeckendorf Games. In the following theorems, we investigate several interesting types of alliances and consider when a player or team is guaranteed a winning strategy.
\begin{thm} \label{DI1}
When $n \geq 3c^2 + 6c + 3$, for any $p \geq c + 2$, no player has a winning strategy in the Multiplayer Generalized Zeckendorf Game. 
\end{thm}

\begin{thm}\label{DI5}
For any $n \geq 2d^2 + 4d$ and $t \geq c+2$, if each team has exactly $d = t - c$ consecutive players, then no team has a winning strategy in the Team Generalized Zeckendorf Game.
\end{thm}

\begin{thm}\label{DI8}
Let $p \geq 6$ and $c = 1$. If there are two teams, one with $p - 2$ players  and the other with two players, then the larger team has a winning strategy for $n \geq 36$ in the Team Generalized Zeckendorf Game. 
\end{thm}
\section{The Multiplayer Tribonacci Game}
We start our investigation of $(c, k)$-nacci games with the Tribonacci game, which is the $(1, 2)$-nacci game. We define the Tribonacci numbers as $T_{1} = 1$, $T_{2} = 2$, $T_{3} = 4$, $T_{i+1} = T_{i} + T_{i-1} + T_{i-2}$. The Tribonacci game is playable since it is a type of $(c, k)$-nacci game, and the types of moves allowed follow the form described for Generalized Zeckendorf games with $c = 1$ and $k = 2$. We first explore winning strategies of the multiplayer Tribonacci Game.\\
\\
\textbf{Theorem \ref{thm1}.}
\textit{
For the multiplayer Tribonacci game, when $n \geq 7$, for any $p \geq 3$ no player has a winning strategy.
}

In general, in a finite multiplayer game consisting of $p\geq3$ players, if Player $k$ has a winning strategy (where $1\leq k\leq p$), then no matter which steps the other $p-1$ players take, there is always a combination of moves for which Player $m$ wins this game.

The following proof of Theorem \ref{thm1} uses a technique which we call a stealing strategy. A stealing strategy means that if we first suppose some Player $k$ has a winning strategy (where $1\leq k \leq p$), then the other $p-1$ players can consider two combinations of moves of different lengths that finish at the same position. For one such combination one of the $p-1$ players will end up in the same position that Player $k$ had a winning strategy, and by stealing Player $k$'s winning strategy we get a contradiction. This technique is used in proving many Theorems and Lemmas of this paper, and we now use it to prove Theorem \ref{thm1}.

\begin{proof}
Note: In all the following proofs of this section, Player $0 \equiv$ Player $p$ $($under $\bmod$ $p)$, and the player following Player $p$ is Player $1$.

To prove Theorem \ref{thm1}, we introduce the following property.

\begin{pro}\label{TribProp}
Suppose Player $m$ has a winning strategy $(1\leq m \leq p)$. For any $p\geq 3$, any winning path of Player $m$ does not contain the following $3$ consecutive steps unless Player $m$ is the player who takes step $2$ below.
\end{pro}
\begin{equation}\nonumber
\begin{split}
\text {Step 1:} & \ 1 \wedge 1 \rightarrow 2.\\
\text {Step 2:} & \ 1 \wedge 1 \rightarrow 2.\\
\text {Step 3:} & \ 2 \wedge 2 \rightarrow 4.\\
\end{split}
\end{equation}

\begin{proof}
Suppose Player $m$ has a winning strategy and there is a winning path that contains these $3$ consecutive steps. Then there exists a Player $a$ where $1\leq a\leq p,\ a \neq m$, such that Player $a-1 \pmod p$ can take step $1$, Player $a$ can take step $2$ and Player $a+1 \pmod p$ can take step $3$.

Note that instead of doing $\{1 \wedge 1 \rightarrow 2\}$, Player $a$ can do $\{1 \wedge 1 \wedge 2 \rightarrow 4\}$. Then Player $m-1 \pmod p$ has a winning strategy, which is a contradiction.

Therefore, by using the stealing strategy, Property \ref{TribProp} holds.
\end{proof}

We now prove Theorem \ref{thm1} by splitting it into the following two lemmas.

\begin{lem}\label{lem2.1.1}
When $n \geq 11$, for any $p \geq 4$ no player has a winning strategy.
\end{lem} 

\begin{proof}
Suppose Player $m$ has a winning strategy $(1\leq m\leq p)$. Consider the following two cases.

Case 1: If $m\geq 4$, then players $1, 2$, and $3$ can do the following.
\begin{equation}\nonumber
\begin{split}
\text {Player 1:} & \ 1 \wedge 1 \rightarrow 2.\\
\text {Player 2:} & \ 1 \wedge 1 \rightarrow 2.\\
\text {Player 3:} & \ 2 \wedge 2 \rightarrow 4.\\
\end{split}
\end{equation}
This contradicts Property \ref{TribProp}, so Player $m$ does not have winning strategy for any $m\geq 4$.

Case 2: If $m\leq 3$, then after Player $m$'s first move, players $m+1, m+2, m+3$ can do the following.
\begin{equation}\nonumber
\begin{split}
\text {Player $m + 1$:} & \ 1 \wedge 1 \rightarrow 2.\\
\text {Player $m + 2$:} & \ 1 \wedge 1 \rightarrow 2.\\
\text {Player $m + 3$:} & \ 2 \wedge 2 \rightarrow 4.\\
\end{split}
\end{equation}
This contradicts Property \ref{TribProp}, so Player $m$ does not have winning strategy for any $m\leq 3$.

By Cases $1$ and $2$, Lemma \ref{lem2.1.1} is proved.
\end{proof}

\begin{lem}\label{lem2.1.2}
When $n \geq 13$, for $p=3$ no player has a winning strategy.
\end{lem}

\begin{proof}
Suppose Player $m$ has a winning strategy $(1\leq m\leq 3)$. After Player $m$'s first move, Players $m+1$ and $m+2$ can do the following $($if $m=3$, we can start the following process from the first step of the game$)$.

\begin{equation}\nonumber
\begin{split}
\text {Step $1$: Player $m + 1$:} & \ 1 \wedge 1 \rightarrow 2.\\
\text {Step $2$: Player $m + 2$:} & \ 1 \wedge 1 \rightarrow 2.\\
\text {Step $3$: Player $m$:} & \ \text {Player $m$ can do anything.}\\
\end{split}
\end{equation}
Note that if Player $m$ does $\{2 \wedge 2 \rightarrow 4\}$, then these three moves violate Property \ref{TribProp}, which is a contradiction.

If Player $m$ does anything else other than $\{2 \wedge 2 \rightarrow 4\}$, then after Player $m$'s first move, the other two players can do the following (continuing after the first $3$ steps listed above with $2$ more steps; if $m=3$, Player $m+1$ is Player $1$).
\begin{equation}\nonumber
\begin{split}
\text {Step $1$: Player $m + 1$:} & \ 1 \wedge 1 \rightarrow 2.\\
\text {Step $2$: Player $m + 2$:} & \ 1 \wedge 1 \rightarrow 2.\\
\text {Step $3$: Player $m$:} & \ \text {player $m$ can do anything.}\\
\text {Step $4$: Player $m + 1$:} & \ 1 \wedge 1 \rightarrow 2.\\
\text {Step $5$: Player $m + 2$:} & \ 2 \wedge 2 \rightarrow 4.\\
\end{split}
\end{equation}
Note that Step $3$ removes at most one $2$, but Step $1$ and Step $2$ generate two $2$'s in total, so there will be at least one $2$ remaining after step $3$. Therefore, Player $m+1$ can do $\{1 \wedge 1 \wedge 2 \rightarrow 4\}$ instead in Step $4$. By doing so, now Player $m-1\pmod p$ has winning strategy, which is a contradiction.

Thus by using the stealing strategy, Lemma \ref{lem2.1.2} is proved.
\end{proof}

By Lemmas \ref{lem2.1.1} and \ref{lem2.1.2}, and brute force computations for $7\leq n\leq 12$, Theorem \ref{thm1} is proved.
\end{proof}

\section{Winning Strategies for Two-Player Generalized Zeckendorf Games}

Expanding off our work on the Tribonacci Game, we now prove several results on the more general $(c,k)$-nacci Zeckendorf Game. We start with proving the winning strategy of two-player $(c,k)$-nacci Zeckendorf Game.\\
\\
\textbf{Theorem \ref{thm3}.}
\textit{
In a $(c,k)$-nacci game, for any $c\geq 1, k \geq 1, n\geq (c+1)^3+(c+1)$, when $c$ is odd, player $2$ always has a winning strategy; when $c$ is even, player $1$ always has a winning strategy.
}

\begin{proof}
To help the reader better understand this theorem and why the value of $k$ does not affect which player has the winning strategy, we first prove a special case of this theorem for the Tribonacci Game as described in the following lemma.\\
\\
\textbf{Lemma \ref{2pTrib}.}
\textit{
For all $n > 9$, Player $2$ has the winning strategy for the two-player Tribonacci Game.
}

\begin{figure}
    \centering
    \begin{tikzcd}[ampersand replacement=\&]
	player \& turn \\
	1 \& 1 \&\& \textcolor{rgb,255:red,214;green,92;blue,92}{1^n} \\
	2 \& 2 \&\& \textcolor{rgb,255:red,214;green,92;blue,92}{1^{(n-2)} \wedge 2} \&\& {} \\
	1 \& 3 \&\& \textcolor{rgb,255:red,214;green,92;blue,92}{\ 1^{(n-4)} \wedge 4} \&\&\& {} \\
	2 \& 4 \& {} \& \textcolor{rgb,255:red,214;green,92;blue,92}{1^{(n-6)} \wedge 2 \wedge 4} \\
	1 \& 5 \& \textcolor{rgb,255:red,214;green,92;blue,92}{1^{(n-7)} \wedge 7} \& {} \& \textcolor{rgb,255:red,214;green,92;blue,92}{1^{(n-8)} \wedge 4^{(2)}} \\
	2 \& 6 \&\& \textcolor{rgb,255:red,214;green,92;blue,92}{1^{(n-10)} \wedge 2 \wedge 4^{(2)}} \& \textcolor{rgb,255:red,92;green,92;blue,214}{1^{(n-7)} \wedge 7} \\
	1 \& 7 \&\& \textcolor{rgb,255:red,92;green,214;blue,92}{1^{(n-9)} \wedge 2 \wedge 7}
	\arrow[no head, from=2-4, to=3-4]
	\arrow[no head, from=3-4, to=4-4]
	\arrow[no head, from=4-4, to=5-4]
	\arrow[no head, from=5-4, to=6-3]
	\arrow[no head, from=5-4, to=6-5]
	\arrow[no head, from=6-5, to=7-4]
	\arrow[no head, from=6-5, to=7-5]
	\arrow[no head, from=7-4, to=8-4]
	\arrow[no head, from=7-5, to=8-4]
    \end{tikzcd}
    \caption{Tribonacci tree depicting the proof of Lemma \ref{2pTrib}. Red indicates a winning strategy for Player $1$, and blue indicates a winning strategy for Player $2$. Green indicates a winning strategy for both players, which is a contradiction.}
    \label{fig:TribTree}
\end{figure}
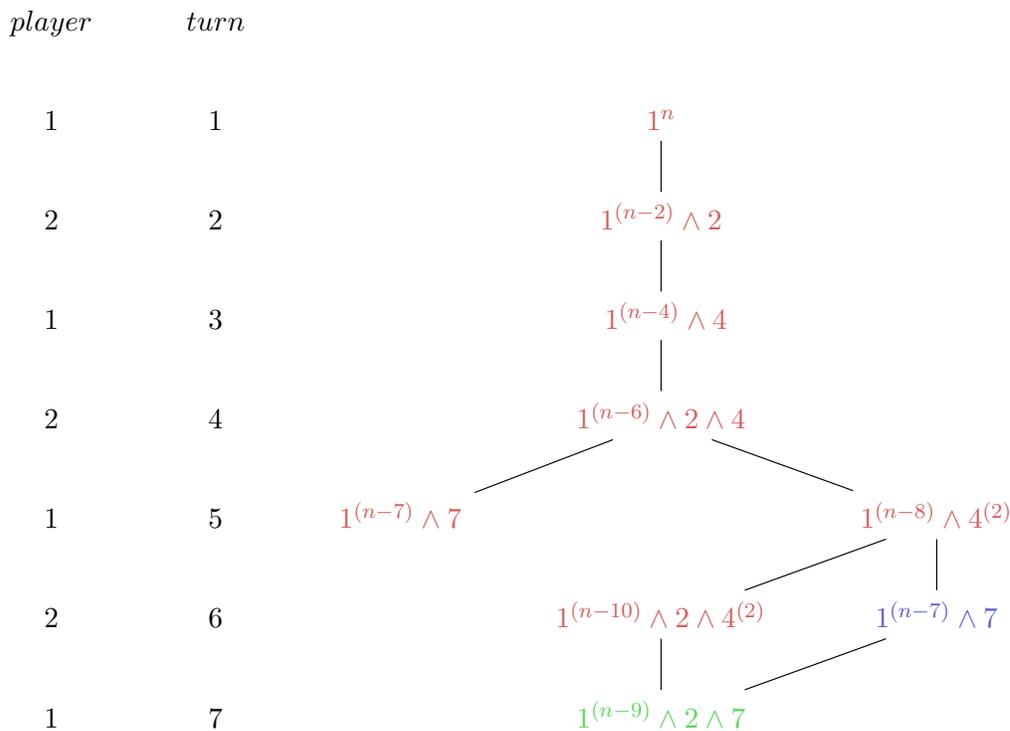

\begin{proof}
Suppose for contradiction that Player $1$ has a winning strategy for $n > 9$. Fix $n > 9$, then $\{1^{(n)}\}$ has a winning strategy for Player $1$ by assumption. Player $1$ must also have a winning strategy for $\{1^{(n-2)} \wedge 2\}$ in row $2$ because it is the only child of $\{1^{(n)}\}$. Since row $2$ is Player $2$'s turn but Player $1$ has a winning strategy, both nodes in row $3$ must also have winning strategies for Player $1$. In row $4$, $\{1^{(n-6)} \wedge 2 \wedge 4\}$ must have a winning strategy for Player $1$ since it is the only child of $\{1^{(n-4)} \wedge 4\}$ in row $3$. Since row $4$ is Player $2$'s turn but Player $1$ has a winning strategy, both children must have winning strategies for Player $1$. This includes $\{1^{(n-7)} \wedge 7\}$ in row $5$, so the equivalent node in row $6$ must have a winning strategy for Player $2$, and its only child $\{1^{(n-9)} \wedge 2 \wedge 7\}$ in row $7$ must also have a winning strategy for Player $2$. In row $6$, $\{1^{(n-10)} \wedge 2 \wedge 4^{(2)}\}$ must have a winning strategy for Player $1$ since at least one child of $\{1^{(n-8)} \wedge 4^{(2)}\}$ in row $5$ must have a winning strategy for Player $1$. However, since row $6$ is Player $2$'s turn and Player $1$ has a winning strategy for $\{1^{(n-10)} \wedge 2 \wedge 4^{(2)}\}$, its child $\{1^{(n-9)} \wedge 2 \wedge 7\}$ in row $7$ must also have a winning strategy for Player $1$. This is a contradiction, so Player $2$ has a winning strategy for $n > 9$. 
\end{proof}

We now prove Theorem \ref{thm3} using the following three lemmas for the cases $k=1$, $k=2$, and $k \geq 3$. In all of the proofs below, we let $d=c-1$.

\begin{lem}\label{(c,1)}
In a $(c,1)$-nacci game, for any $c\geq 1, n\geq (c+1)^3+(c+1)$, when $c$ is odd, Player $2$ always has a winning strategy; when $c$ is even, Player $1$ always has a winning strategy.
\end{lem}

\begin{proof}
We consider an equivalent statement: In a $(d, 1)$-nacci game, where $d=c-1$, for any $c\geq 2, n\geq (c)^3+(c)$, when $d$ is odd, Player $2$ always has a winning strategy; when $d$ is even, Player $1$ always has a winning strategy. We first consider the case when $d$ is odd, and we suppose for contradiction that player $1$ has a winning strategy. From row $1$ to row $c$, each node only has one child, so player $1$ has a winning strategy for all these $c$ rows. Row $c$ is Player $2$'s turn but player $1$ has a winning strategy for $\{1^{(n-c^2+c)} \wedge c^{(c-1)}\}$, so all of its children in row $c+1$ have a winning strategy for Player $1$. The node $\{1^{(n-c^2-c+1)} \wedge c \wedge cc-1\}$ in row $c+2$ is the only child of $\{1^{(n-c^2+1)} \wedge cc-1\}$ in row $c+1$, so it must also have a winning strategy for player $1$. From row $c+2$ to row $2c$, each node only has one child, so Player $1$ has a winning strategy for all these rows. We can repeat the steps used from row $c+1$ to row $2c$ to show that Player $1$ has a winning strategy until $\{1^{(n-c^3+2c-1)} \wedge c^{(c-1)} \wedge (cc-1)^{(c-1)}\}$ in row $c^2$. Row $c^2$ is Player $2$'s turn, so player $1$ has a winning strategy for all of its children in row $c^2+1$. Since $\{1^{(n-c^3+2c-1)} \wedge ccc-2c+1\}$ in row $c^2+1$ has only one child, $\{1^{(n-c^3+2c-1)} \wedge c \wedge ccc-2c+1\}$ in row $c^2+2$ must also have a winning strategy for Player $1$. Then Player $2$ must have a winning strategy for the equivalent node in row $C^2+3$. However, since player $1$ has a winning strategy for $\{1^{(n-c^3-c)} \wedge (cc-1)^{(c)}\}$ in row $c^2+1$, at least one of its children must also have a winning strategy for player $1$. Both children are parents to the node $\{1^{(n-c^3+c-1)} \wedge c \wedge ccc-2c+1\}$ in row $c^2+3$, and since row $c^2+2$ is Player $2$'s turn, $\{1^{(n-c^3+c-1)} \wedge c \wedge ccc-2c+1\}$ in row $c^2+3$ must also have a winning strategy for Player $1$. This is a contradiction, so Player $2$ must have a winning strategy.

A similar proof also applies to the case when $d$ is even, which is shown in Appendix \ref{appendix 1}.

\end{proof}

\begin{figure}
    \centering
\resizebox{18cm}{!}{
\begin{tikzcd}[ampersand replacement=\&]
	player \& turn \\
	1 \& 1 \&\& \textcolor{rgb,255:red,214;green,92;blue,92}{1^n} \\
	2 \& 2 \&\& \textcolor{rgb,255:red,214;green,92;blue,92}{1^{(n-c)} \wedge c} \&\& {} \\
	2 \& c \&\& \textcolor{rgb,255:red,214;green,92;blue,92}{\ 1^{(n-c^2+c)} \wedge c^{(c-1)}} \&\&\& {} \\
	1 \& {c+1} \& \textcolor{rgb,255:red,214;green,92;blue,92}{1^{(n-c^2)} \wedge c^{(c)}} \&\& \textcolor{rgb,255:red,214;green,92;blue,92}{1^{(n-c^2+1)} \wedge cc-1} \&\&\\
	2 \& {c+2} \&\& \textcolor{rgb,255:red,214;green,92;blue,92}{1^{(n-c^2-c+1)} \wedge c \wedge cc-1} \\
	2 \& 2c \&\& \textcolor{rgb,255:red,214;green,92;blue,92}{1^{(n-2c^2+c+1)} \wedge c^{(c-1)} \wedge cc-1} \\
	1 \& {2c+1} \& \textcolor{rgb,255:red,214;green,92;blue,92}{1^{(n-2c^2+1)} \wedge c^{(c)} \wedge cc-1} \& \textcolor{rgb,255:red,214;green,92;blue,92}{1^{(n-2c^2+2)} \wedge (cc-1)^{(2)}} \&\&\\
	2 \& {2c+2} \&\& \textcolor{rgb,255:red,214;green,92;blue,92}{1^{(n-2c^2-c+2)} \wedge c \wedge (cc-1)^{(2)}} \\
	2 \& {c^2} \&\& \textcolor{rgb,255:red,214;green,92;blue,92}{1^{(n-c^3+2c-1)} \wedge c^{(c-1)} \wedge (cc-1)^{(c-1)}} \\
	1 \& {c^2+1} \& \textcolor{rgb,255:red,214;green,92;blue,92}{1^{(n-c^3+c-1)} \wedge c^{(c)} \wedge (cc-1)^{(c-1)}} \& \textcolor{rgb,255:red,214;green,92;blue,92}{1^{(n-c^3-c)} \wedge (cc-1)^{(c)}} \& \textcolor{rgb,255:red,214;green,92;blue,92}{1^{(n-c^3+2c-1)} \wedge ccc-2c+1} \\
	2 \& {c^2+2} \& \textcolor{rgb,255:red,214;green,92;blue,92}{1^{(n-c^3)} \wedge c \wedge (cc-1)^{(c)}} \& \textcolor{rgb,255:red,214;green,92;blue,92}{1^{(n-c^3+2c-1)} \wedge ccc-2c+1} \& \textcolor{rgb,255:red,214;green,92;blue,92}{1^{(n-c^3+c-1)} \wedge c \wedge ccc-2c+1} \\
	1 \& {c^2+3} \&\& \textcolor{rgb,255:red,92;green,214;blue,92}{1^{(n-c^3+c-1)} \wedge c \wedge ccc-2c+1} \& {}
	\arrow[no head, from=2-4, to=3-4]
	\arrow[dashed, no head, from=3-4, to=4-4]
	\arrow[no head, from=4-4, to=5-3]
	\arrow[no head, from=4-4, to=5-5]
	\arrow[no head, from=5-3, to=6-4]
	\arrow[no head, from=5-5, to=6-4]
	\arrow[dashed, no head, from=6-4, to=7-4]
	\arrow[no head, from=7-4, to=8-3]
	\arrow[no head, from=7-4, to=8-4]
	\arrow[no head, from=10-4, to=11-3]
	\arrow[no head, from=8-4, to=9-4]
	\arrow[no head, from=8-3, to=9-4]
	\arrow[dashed, no head, from=9-4, to=10-4]
	\arrow[no head, from=10-4, to=11-4]
	\arrow[no head, from=10-4, to=11-5]
	\arrow[no head, from=11-3, to=12-3]
	\arrow[no head, from=11-4, to=12-3]
	\arrow[no head, from=11-4, to=12-4]
	\arrow[no head, from=11-5, to=12-5]
	\arrow[no head, from=12-3, to=13-4]
	\arrow[no head, from=12-4, to=13-4]
	\arrow[dashed, no head, from=3-1, to=4-1]
	\arrow[dashed, no head, from=3-2, to=4-2]
	\arrow[dashed, no head, from=6-2, to=7-2]
	\arrow[dashed, no head, from=6-1, to=7-1]
	\arrow[dashed, no head, from=9-2, to=10-2]
	\arrow[dashed, no head, from=9-1, to=10-1]
\end{tikzcd}
}
    \caption{$(d, 1)$-nacci tree depicting the proof of Theorem \ref{(c,1)}, where $d=c-1$ is odd. Red indicates a winning strategy for Player $1$, and blue indicates a winning strategy for Player $2$. Green indicates a winning strategy for both players, which is a contradiction.}
    \label{fig:(c,1)-Tree}
\end{figure}

\begin{lem}\label{(c,2)}
In a $(c,2)$-nacci game, for any $c\geq 1, n\geq (c+1)^3+(c+1)$, when $c$ is odd, player $2$ always has a winning strategy; when $c$ is even, player $1$ always has a winning strategy.
\end{lem}

\begin{proof}
The proof of Lemma \ref{(c,2)} follows the same format of the proof of Lemma \ref{(c,1)}, but the nodes are slightly different because of the different $k$-value. 
\end{proof}

\begin{figure}
    \centering
\resizebox{15cm}{!}{
\begin{tikzcd}[ampersand replacement=\&]
	player \& turn \\
	1 \& 1 \&\& \textcolor{rgb,255:red,214;green,92;blue,92}{1^n} \\
	2 \& 2 \&\& \textcolor{rgb,255:red,214;green,92;blue,92}{1^{(n-c)} \wedge c} \\
	2 \& c \&\& \textcolor{rgb,255:red,214;green,92;blue,92}{1^{(n-c^2+c)} \wedge c^{(c-1)}} \\
	1 \& {c+1} \& \textcolor{rgb,255:red,214;green,92;blue,92}{1^{(n-c^2)} \wedge c^{(c)}} \&\& \textcolor{rgb,255:red,214;green,92;blue,92}{1^{(n-c^2)} \wedge cc} \\
	2 \& {c+2} \&\& \textcolor{rgb,255:red,214;green,92;blue,92}{1^{(n-c^2-c)} \wedge c \wedge cc} \\
	2 \& 2c \&\& \textcolor{rgb,255:red,214;green,92;blue,92}{1^{(n-2c^2+c)} \wedge c^{(c-1)} \wedge cc} \\
	1 \& {2c+1} \& \textcolor{rgb,255:red,214;green,92;blue,92}{1^{(n-2c^2)} \wedge c^{(c)} \wedge cc} \& \textcolor{rgb,255:red,214;green,92;blue,92}{1^{(n-2c^2)} \wedge cc^{(2)}} \\
	2 \& {2c+2} \&\& \textcolor{rgb,255:red,214;green,92;blue,92}{1^{(n-2c^2-c)} \wedge c \wedge cc^{(2)}} \\
	2 \& {c^2} \&\& \textcolor{rgb,255:red,214;green,92;blue,92}{1^{(n-c^3+c)} \wedge c^{(c-1)} \wedge cc^{(c-1)}} \\
	1 \& {c^2+1} \& \textcolor{rgb,255:red,214;green,92;blue,92}{1^{(n-c^3)} \wedge c^{(c)} \wedge cc^{(c-1)}} \& \textcolor{rgb,255:red,214;green,92;blue,92}{1^{(n-c^3)} \wedge cc^{(c)}} \& \textcolor{rgb,255:red,214;green,92;blue,92}{1^{(n-c^3+1)} \wedge ccc-1} \\
	2 \& {c^2+2} \& \textcolor{rgb,255:red,214;green,92;blue,92}{1^{(n-c^3-c)} \wedge c \wedge cc^{(c)}} \& \textcolor{rgb,255:red,214;green,92;blue,92}{1^{(n-c^3+1)} \wedge ccc-1} \& \textcolor{rgb,255:red,214;green,92;blue,92}{1^{(n-c^3-c+1)} \wedge c \wedge ccc-1} \\
	1 \& {c^2+3} \&\& \textcolor{rgb,255:red,92;green,214;blue,92}{1^{(n-c^3-c+1)} \wedge c \wedge ccc-1} \\
	\&\&\&\& {}
	\arrow[no head, from=2-4, to=3-4]
	\arrow[dashed, no head, from=3-4, to=4-4]
	\arrow[no head, from=4-4, to=5-3]
	\arrow[no head, from=4-4, to=5-5]
	\arrow[no head, from=5-3, to=6-4]
	\arrow[no head, from=5-5, to=6-4]
	\arrow[dashed, no head, from=6-4, to=7-4]
	\arrow[no head, from=7-4, to=8-4]
	\arrow[no head, from=7-4, to=8-3]
	\arrow[no head, from=8-4, to=9-4]
	\arrow[no head, from=8-3, to=9-4]
	\arrow[dashed, no head, from=9-4, to=10-4]
	\arrow[no head, from=10-4, to=11-4]
	\arrow[no head, from=10-4, to=11-3]
	\arrow[no head, from=10-4, to=11-5]
	\arrow[no head, from=11-3, to=12-3]
	\arrow[no head, from=11-4, to=12-3]
	\arrow[no head, from=11-5, to=12-5]
	\arrow[no head, from=12-3, to=13-4]
	\arrow[no head, from=11-4, to=12-4]
	\arrow[no head, from=12-4, to=13-4]
	\arrow[dashed, no head, from=3-2, to=4-2]
	\arrow[dashed, no head, from=3-1, to=4-1]
	\arrow[dashed, no head, from=6-2, to=7-2]
	\arrow[dashed, no head, from=6-1, to=7-1]
	\arrow[dashed, no head, from=9-2, to=10-2]
	\arrow[dashed, no head, from=9-1, to=10-1]
\end{tikzcd}
}
    \caption{$(d, 2)$-nacci tree depicting the proof of Theorem \ref{(c,2)}, where $d=c-1$ is odd. Red indicates a winning strategy for Player $1$, and blue indicates a winning strategy for Player $2$. Green indicates a winning strategy for both players, which is a contradiction.}
    \label{fig:(c,2)-Tree}
\end{figure}

\begin{lem}\label{(c,k)}
In a $(c,k)$-nacci game, for any $c\geq 1$, $k \geq 3, n\geq (c+1)^3+(c+1)$, when $c$ is odd, Player $2$ always has a winning strategy; when $c$ is even, Player $1$ always has a winning strategy.
\end{lem}

\begin{proof}
The proof of Lemma \ref{(c,k)} follows the same format of the proof of Lemma \ref{(c,1)}, but the nodes are slightly different because of the different $k$-value. We provide a detailed proof in Appendix \ref{appendix 3}.
\end{proof}

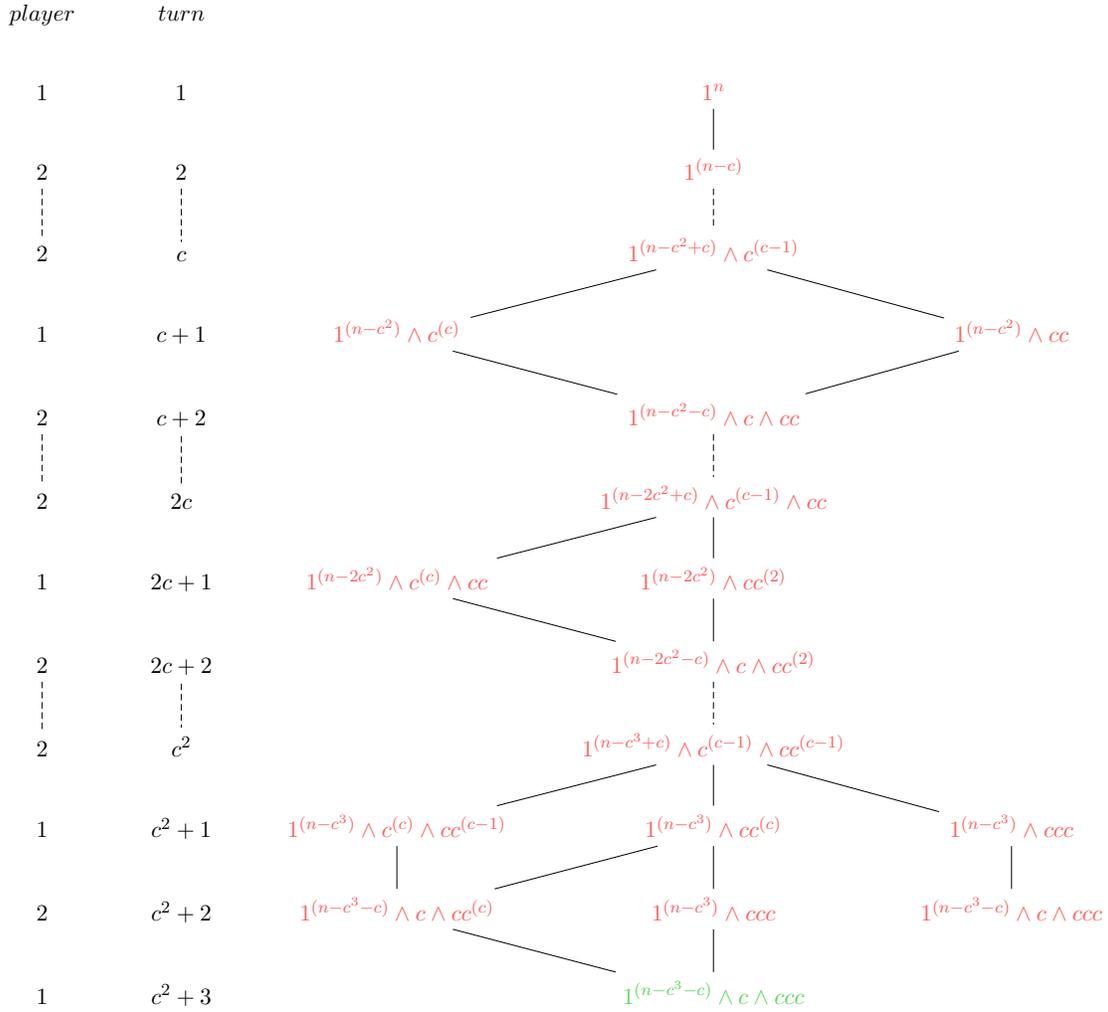
\begin{figure}
    \centering
\resizebox{15cm}{!}{
\begin{tikzcd}[ampersand replacement=\&]
	player \& turn \\
	1 \& 1 \&\& \textcolor{rgb,255:red,255;green,92;blue,92}{1^n} \\
	2 \& 2 \&\& \textcolor{rgb,255:red,255;green,92;blue,92}{1^{(n-c)}} \\
	2 \& c \&\& \textcolor{rgb,255:red,255;green,92;blue,92}{1^{(n-c^2+c)} \wedge c^{(c-1)}} \\
	1 \& {c+1} \& \textcolor{rgb,255:red,255;green,92;blue,92}{1^{(n-c^2)} \wedge c^{(c)}} \&\& \textcolor{rgb,255:red,255;green,92;blue,92}{1^{(n-c^2)} \wedge cc} \\
	2 \& {c+2} \&\& \textcolor{rgb,255:red,255;green,92;blue,92}{1^{(n-c^2-c)} \wedge c \wedge cc} \\
	2 \& 2c \&\& \textcolor{rgb,255:red,255;green,92;blue,92}{1^{(n-2c^2+c)} \wedge c^{(c-1)} \wedge cc} \\
	1 \& {2c+1} \& \textcolor{rgb,255:red,255;green,92;blue,92}{1^{(n-2c^2)} \wedge c^{(c)} \wedge cc} \& \textcolor{rgb,255:red,255;green,92;blue,92}{1^{(n-2c^2)} \wedge cc^{(2)}} \\
	2 \& {2c+2} \&\& \textcolor{rgb,255:red,255;green,92;blue,92}{1^{(n-2c^2-c)} \wedge c \wedge cc^{(2)}} \\
	2 \& {c^2} \&\& \textcolor{rgb,255:red,255;green,92;blue,92}{1^{(n-c^3+c)} \wedge c^{(c-1)} \wedge cc^{(c-1)}} \\
	1 \& {c^2+1} \& \textcolor{rgb,255:red,255;green,92;blue,92}{1^{(n-c^3)} \wedge c^{(c)} \wedge cc^{(c-1)}} \& \textcolor{rgb,255:red,255;green,92;blue,92}{1^{(n-c^3)} \wedge cc^{(c)}} \& \textcolor{rgb,255:red,255;green,92;blue,92}{1^{(n-c^3)} \wedge ccc} \\
	2 \& {c^2+2} \& \textcolor{rgb,255:red,255;green,92;blue,92}{1^{(n-c^3-c)} \wedge c \wedge cc^{(c)}} \& \textcolor{rgb,255:red,255;green,92;blue,92}{1^{(n-c^3)} \wedge ccc} \& \textcolor{rgb,255:red,255;green,92;blue,92}{1^{(n-c^3-c)} \wedge c \wedge ccc} \\
	1 \& {c^2+3} \&\& \textcolor{rgb,255:red,92;green,214;blue,92}{1^{(n-c^3-c)} \wedge c \wedge ccc}
	\arrow[no head, from=2-4, to=3-4]
	\arrow[dashed, no head, from=3-4, to=4-4]
	\arrow[no head, from=4-4, to=5-5]
	\arrow[no head, from=4-4, to=5-3]
	\arrow[no head, from=5-5, to=6-4]
	\arrow[no head, from=5-3, to=6-4]
	\arrow[dashed, no head, from=6-4, to=7-4]
	\arrow[no head, from=7-4, to=8-4]
	\arrow[no head, from=7-4, to=8-3]
	\arrow[no head, from=8-4, to=9-4]
	\arrow[no head, from=8-3, to=9-4]
	\arrow[dashed, no head, from=9-4, to=10-4]
	\arrow[no head, from=10-4, to=11-3]
	\arrow[no head, from=10-4, to=11-5]
	\arrow[no head, from=10-4, to=11-4]
	\arrow[no head, from=11-3, to=12-3]
	\arrow[no head, from=11-4, to=12-3]
	\arrow[no head, from=12-4, to=11-4]
	\arrow[no head, from=11-5, to=12-5]
	\arrow[no head, from=12-4, to=13-4]
	\arrow[dashed, no head, from=3-2, to=4-2]
	\arrow[dashed, no head, from=3-1, to=4-1]
	\arrow[dashed, no head, from=6-2, to=7-2]
	\arrow[dashed, no head, from=6-1, to=7-1]
	\arrow[dashed, no head, from=9-2, to=10-2]
	\arrow[dashed, no head, from=9-1, to=10-1]
	\arrow[no head, from=12-3, to=13-4]
\end{tikzcd}
}
    \caption{$(d, k)$-nacci tree depicting the proof of Theorem \ref{(c,k)}, where $d=c-1$ is odd and $k\geq 3$. Red indicates a winning strategy for Player $1$, and blue indicates a winning strategy for Player $2$. Green indicates a winning strategy for both players, which is a contradiction.}
    \label{fig:(c,k)-Tree}
\end{figure}

Theorem \ref{thm3} is proved by Lemmas \ref{(c,1)}, \ref{(c,2)}, and \ref{(c,k)}
\end{proof}

We now prove the following result on the earliest possible turn a player can lose their winning strategy for the $(c,k)$-nacci game.\\
\\
\textbf{Theorem \ref{thm4}.}
\textit{
For the $(c,k)$-nacci game, where $k > 1$, if the player with the winning strategy makes a mistake as early as turn $c + 1$, the opposing player can steal the winning strategy.
}

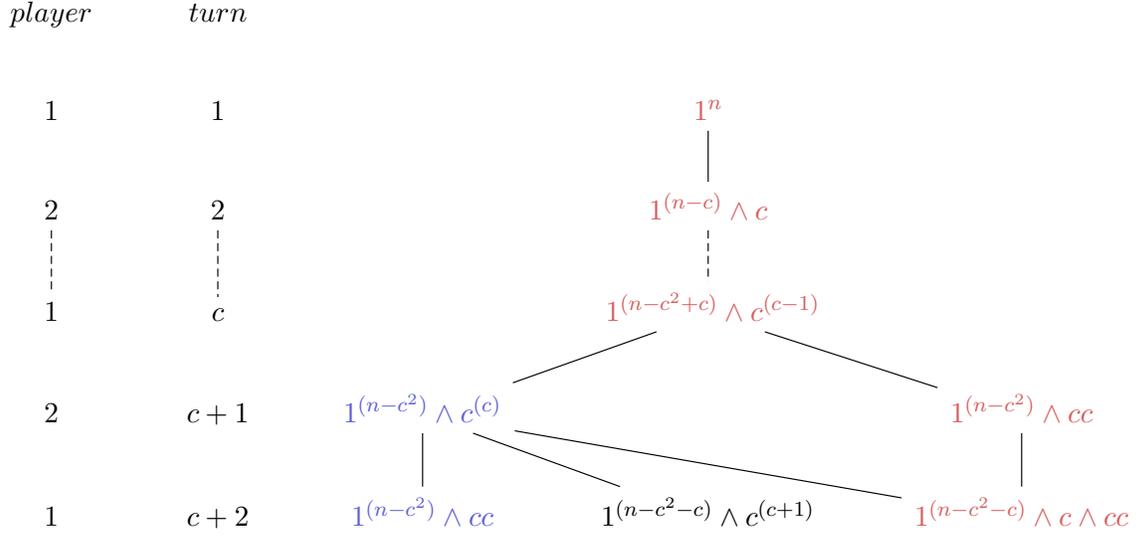
\begin{figure}
    \centering
\begin{tikzcd}[ampersand replacement=\&]
	player \& turn \\
	1 \& 1 \&\& \textcolor{rgb,255:red,214;green,92;blue,92}{1^n} \\
	2 \& 2 \&\& \textcolor{rgb,255:red,214;green,92;blue,92}{1^{(n-c)} \wedge c} \&\& {} \\
	1 \& c \&\& \textcolor{rgb,255:red,214;green,92;blue,92}{\ 1^{(n-c^2+c)} \wedge c^{(c-1)}} \&\& {} \& {} \\
	2 \& {c+1} \& \textcolor{rgb,255:red,92;green,92;blue,214}{1^{(n-c^2)} \wedge c^{(c)}} \&\& \textcolor{rgb,255:red,214;green,92;blue,92}{1^{(n-c^2)} \wedge cc} \\
	1 \& {c+2} \& \textcolor{rgb,255:red,92;green,92;blue,214}{1^{(n-c^2)} \wedge cc} \& {1^{(n-c^2-c)} \wedge c^{(c+1)}} \& \textcolor{rgb,255:red,214;green,92;blue,92}{1^{(n-c^2-c)} \wedge c \wedge cc}
	\arrow[no head, from=2-4, to=3-4]
	\arrow[dashed, no head, from=3-4, to=4-4]
	\arrow[no head, from=4-4, to=5-3]
	\arrow[no head, from=4-4, to=5-5]
	\arrow[no head, from=5-3, to=6-3]
	\arrow[no head, from=5-3, to=6-4]
	\arrow[no head, from=5-3, to=6-5]
	\arrow[no head, from=5-5, to=6-5]
	\arrow[dashed, no head, from=3-1, to=4-1]
	\arrow[dashed, no head, from=3-2, to=4-2]
\end{tikzcd}
    \caption{$(d, k)$-nacci tree depicting the proof of Theorem \ref{thm4}, where $d=c-1$. Red indicates a winning strategy for Player $1$, and blue indicates a winning strategy for Player $2$.}
    \label{fig:MistakeTree}
\end{figure}

\begin{proof}
We prove Theorem \ref{thm4} by splitting it into the following two lemmas.

\begin{lem}\label{CI1}
When $c$ is even Player $1$ has the winning strategy for the $(c, k)$-nacci game, but Player $2$ can steal the winning strategy if Player $1$ makes a mistake as early as turn $c + 1$.
\end{lem}

\begin{proof}
Let $d=c-1$. We have previously shown that when $c$ is even, Player $1$ has the winning strategy for the $(c, k)$-nacci game. Equivalently, let $d=c-1$ be even, then Player $1$ has the winning strategy for the $(d, k)$-nacci game. In this game there is only one possible move for each of the first $d$ turns, so Player $1$ has the winning strategy until $\{1^{(n-c^2+c)} \wedge c^{(c-1)}\}$ in row $c$. This means that at least one of the nodes in row $c + 1$ must have a winning strategy for Player $1$. Suppose $\{1^{(n-c^2)} \wedge c^{(c)}\}$ in row $c + 1$ has a winning strategy for Player $1$, then since it is Player $2$'s turn, all three of its children in row $c + 2$ must have a winning strategy for Player $1$. Since $\{1^{(n-c^2+c)} \wedge c \wedge cc\}$ in row $c + 2$ has a winning strategy for Player $1$, its parent $\{1^{(n-c^2)} \wedge cc\}$ in row $c + 1$ must as well. However, Player $2$ can steal Player $1$'s winning strategy from the equivalent node in row $c + 2$, so $\{1^{(n-c^2)} \wedge cc\}$ in row $c + 1$ also has a winning strategy for Player $2$. From this contradiction, $\{1^{(n-c^2)} \wedge c^{(c)}\}$ in row $c + 1$ must have a winning strategy for Player $2$. Therefore, if Player $1$ makes a mistake as early as turn $c$, Player $2$ can steal the winning strategy. This analysis has been for the $(c - 1, k)$-nacci game, so in the $(c, k)$-nacci game, Player $2$ can steal the winning strategy as early as turn $c + 1$.
\end{proof}

\begin{lem}\label{CI2}
When $c$ is odd Player $2$ has the winning strategy for the $(c, k)$-nacci game, but Player $1$ can steal the winning strategy if player $2$ makes a mistake as early as turn $c + 1$.
\end{lem}

\begin{proof}
Once again we consider the $(c - 1, k)$-nacci game, but now $c$ is even and Player $2$ has the winning strategy. However, since turn $c$ is now Player $2$'s turn, the remainder of the proof is identical to the proof of Lemma \ref{CI1}, but with the positions switched. 
\end{proof}
By Lemmas \ref{CI1} and \ref{CI2}, Theorem \ref{thm4} is proved.
\end{proof}

\section{Multiplayer and Multialliance Generalized Zeckendorf Games}

After working on the winning strategy of two-player $(c,k)$-nacci Zeckendorf Game, we now move to specific cases of the multiplayer and multialliance $(c,k)$-nacci Game, and our main results include Theorem \ref{DI1}, Theorem \ref{DI5}, and Theorem \ref{DI8}. These results are extensions of previous theorems on the Multiplayer Zeckendorf Game that is played on the Fibonacci numbers.\\
\\
\textbf{Theorem \ref{DI1}.}
\textit{
When $n \geq 3c^2 + 6c + 3$, for any $p \geq c + 2$, no player has a winning strategy in the Multiplayer Generalized Zeckendorf Game. 
}

\begin{proof}
Note: In all the following proofs of this section, Player $0$ = Player $p$ (under mod $p$).

To prove Theorem \ref{DI1}, we introduce the following property.

\begin{pro}\label{DI2}
Suppose Player $m$ has a winning strategy $(1 \leq m \leq p)$. For any $p \geq 3$, if Player $m$ is not the player who takes step $c + 1$ listed below, then any winning path of player $m$ does not contain the following $c+2$ consecutive steps. \\
\begin{equation}
\begin{split}
\text {Step 1:} & \ (c+1)S_{1} \rightarrow S_{2}.\\
\text {Step 2:} & \ (c+1)S_{1} \rightarrow S_{2}.\\
\vdots \\
\text {Step c + 1:} & \ (c+1)S_{1} \rightarrow S_{2}.\\
\text {Step c + 2:} & \
\begin{cases}
(c+1)S_{2} \rightarrow S_{3} \wedge S_{1} & \text {if k = 1}\\
(c+1)S_{2} \rightarrow S_{3} & \text {if k > 1}.
\end{cases}
\end{split}
\end{equation}

\begin{proof}
Suppose Player $m$ has a winning strategy and there is a winning path that contains these $c+2$ consecutive steps. Then there exists a Player $a$ where $1 \leq a \leq p$, $a \neq m$, such that Player $a - c$ (mod $p$) can take step 1, Player $a - c +1$ (mod $p$) can take step 2, continuing in this way until Player $a$ takes step $c + 1$ and Player $a + 1$ (mod $p$) can take step $c + 2$.

Note that if $k = 1$, Player $a$ can instead do $\{c S_{1} \wedge c S_{2} \rightarrow S_{3}\}$, and if $k > 1$, Player $a$ can instead do $\{ (c + 1) S_{1} \wedge c S_{2} \rightarrow S_{3} \}$. Then Player $m - 1$ (mod $p$) has a winning strategy, which is a contradiction.
\end{proof}
\end{pro}

We now prove Theorem \ref{DI1} by splitting it into the following two lemmas.

\begin{lem}\label{DI3}
When $n \geq 2c^2 + 5c +3$, for any $p \geq c + 3$ no player has a winning strategy.
\end{lem}

\begin{proof}
Suppose Player $m$ has a winning strategy $(1 \leq m \leq p)$. 
Consider the following two cases. \\

Case 1: If $m \geq c + 3$, then Players 1, 2, $\ldots$, $c + 1$, $c + 2$ can do the sequence of $c+2$ moves described in Property \ref{DI2}. This contradicts Property \ref{DI2}, so player $m$ does not have a winning strategy for any
$m \geq c + 3$. \\

Case 2: If $m \leq c + 2$, then after Player $m$'s first move, Players $m + 1$, $m + 2$, $\ldots$, $m + c + 2$ (mod $p$) can do the sequence of $c+2$ moves described in Property \ref{DI2}. This contradicts Property \ref{DI2}, so player $m$ does not have a winning strategy for any
$m \leq c + 2$. \\

Case 2 requires a maximum of $2c + 3$ steps of the move  $\{(c+1)S_{1} \rightarrow S_{2}\}$, which requires that $n \geq 2c^2 + 5c + 3$.
By Case 1 and Case 2, Lemma \ref{DI3} is proved.
\end{proof}

\begin{lem}\label{DI4}
When $n \geq 3c^2 + 6c +3$, for $p = c + 2$ no player has a winning strategy.
\end{lem}

\begin{proof}
Suppose Player $m$ has a winning strategy $(1 \leq m \leq c + 2)$. On their first moves, Players $m + 1$, $m + 2$, $\ldots$, $m + c +1$ (mod $p$) can do the following sequence of moves. \\
\begin{equation}
\begin{split}
\text {Step 1: Player m + 1:} & \ (c+1)S_{1} \rightarrow S_{2}. \\
\vdots \\
\text {Step c + 1: Player m + c + 1:} & \ (c+1)S_{1} \rightarrow S_{2}.\\
\text {Step c + 2: Player m:} & \text { Player m can do anything}.
\end{split}
\end{equation}

Note that if Player $m$ does the move $\{(c+1)S_{2} \rightarrow S_{3} \wedge S_{1}\}$ if $k=1$, or $\{(c+1)S_{2} \rightarrow S_{3}\}$ if $k > 1$, then steps 1 through $c+2$ violate Property \ref{DI2}, which is a contradiction. Player $m+1$ can then continue with the following sequence of moves. \\
\begin{equation}
\begin{split}
\text {Step c + 3: Player m + 1:} & \ (c+1)S_{1} \rightarrow S_{2}. \\
\vdots \\
\text {Step 2c + 2: Player m + c:} & \ (c+1)S_{1} \rightarrow S_{2}. \\
\text {Step 2c + 3: Player m + c +1:} & \
\begin{cases}
(c+1)S_{2} \rightarrow S_{3} \wedge S_{1} & \text {if k = 1}\\
(c+1)S_{2} \rightarrow S_{3} & \text {if k > 1}.
\end{cases}
\end{split}
\end{equation}

Note that step $c + 2$ removes at most $c$ of $S_{1}$, but steps 1 through $c + 1$ generate $c + 1$ of $S_{1}$, so there will be at least one remaining after step $c + 2$. However, on step $2c + 2$ Player $m + c$ can instead do $\{c S_{1} \wedge c S_{2} \rightarrow S_{3} \}$ if $k = 1$, or $\{(c + 1) S_{1} \wedge c S_{2} \rightarrow c S_{3} \}$ if $k > 1$. By doing so Player $m - 1$ (mod $p$) now has a winning strategy which is a contradiction. 

The steps described above requires a maximum of $3c + 3$ steps of the move  $\{(c+1)S_{1} \rightarrow S_{2}\}$, which requires that $n \geq 3c^2 + 6c + 3$. This proves that for $p = c + 2$ no player has a winning strategy when $n \geq 3c^2 + 6c + 3$.
\end{proof}

Theorem \ref{DI1} is proved by Lemmas \ref{DI2} and \ref{DI3}.\\
\end{proof}

\textbf{Theorem \ref{DI5}.}
\textit{
For any $n \geq 2d^2 + 4d$ and $t \geq c+2$, if each team has exactly $d = t - c$ consecutive players, then no team has a winning strategy in the Team Generalized Zeckendorf Game.
}

\begin{proof}
For the following proofs, team 0 = team $t$ (under mod $t$). Note that after player $td$ the next player to move is player 1, and we regard player $td$ and player 1 as two
consecutive players. Therefore, without loss of generality, in all the following proofs, we assume
that team 1 has players $1, 2, 3, \ldots, d$; team 2 has players $d + 1, d + 2, \ldots, 2d$; team 3 has players
$2d + 1, 2d + 2, \ldots, 3d$ and so on. 

To prove Theorem \ref{DI5}, we introduce the following property.

\begin{pro} \label{property 2}
Suppose team $m$ has a winning strategy $(1 \leq m \leq t)$. For any $t \geq c + 2$ and $d = t - c$, if none of the middle $d$ players listed below belong to team $m$, then any winning path for team $m$ does not contain the following 3$d$ (for case 1) or $(c+2)d$ (for case 2) consecutive steps. \\
Case 1: Let $c = 1$.
\begin{equation}
\begin{split}
\text {First d players all do:} & \ (2)S_{1} \rightarrow S_{2}.\\
\text {Middle d players all do:} & \ (2)S_{1} \rightarrow S_{2}.\\
\text {Last d players all do:} & \
\begin{cases}
(2)S_{2} \rightarrow S_{3} \wedge S_{1} & \text {if k = 1}\\
(2)S_{2} \rightarrow S_{3} & \text {if k > 1}.
\end{cases}
\end{split}
\end{equation}
Case 2: Let $c > 1$.
\begin{equation}
\begin{split}
\text {First c players all do:} & \ (c+1)S_{1} \rightarrow S_{2}. \\
\text {Player c + 1:} & \ (c+1)S_{1} \rightarrow S_{2}.\\
\text {Player c + 2:} & \
\begin{cases}
(c+1)S_{2} \rightarrow S_{3} \wedge S_{1} & \text {if k = 1}\\
(c+1)S_{2} \rightarrow S_{3} & \text {if k > 1}.
\end{cases}
\end{split}
\end{equation}
Repeat this sequence of $c + 2$ steps $d$ times.
\end{pro}
\begin{proof}
Suppose team $m$ has a winning strategy and there is a winning path for team $m$ that contains these 3$d$ (for case 1) or $(c + 2)d$ (for case 2) consecutive steps. Then there exists a Player $q$ $(1 \leq q \leq p)$ that belongs to team $m$ and takes the last move of the game. \\

Case 1: The middle $d$ players, instead of doing the move $\{(2)S_{1} \rightarrow S_{2}\}$, can instead do the following: \\
$\begin{cases}
S_{1} \wedge S_{2} \rightarrow S_{3}  & \text {if k = 1}\\
(2)S_{1} \wedge S_{2} \rightarrow S_{3} & \text {if k > 1}.
\end{cases}$ \\
By doing so Player $q-d$ becomes the player who takes the last move. Since team $m$ has $d$ players, Player $q-d$ belongs to team $m-1$ (mod t). Therefore team $m-1$ (mod t) has the winning strategy which is a contradiction. \\

Case 2: Instead of doing the move $\{(c + 1)S_{1} \rightarrow S_{2}\}$, player $c + 1$ can instead do the following: \\
$\begin{cases}
c S_{1} \wedge c S_{2} \rightarrow S_{3} & \text {if k = 1}\\
(c + 1)S_{1} \wedge c S_{2} \rightarrow S_{3} & \text {if k > 1}.
\end{cases}$ \\
By doing so Player $q-1$ becomes the player to take the last move. There are $t d$ total players and since $t \geq c+2$ there are at least $(c+2)d$ total players. When the sequence of moves in case 2 is repeated $d$ times, Player $q-d$ becomes the player to take the last move, which is a contradiction.
\end{proof}

We now prove Theorem \ref{DI5} by splitting it into the following two lemmas.

\begin{lem}\label{DI6}
For any $t \geq c + 3$ and $d=t-c$, no team has a winning strategy when $n \geq 3c^2+15c+12$ and $t \geq 2c$.
\end{lem}

\begin{proof}
Suppose team $m$ has a winning strategy $(1 \leq m \leq t)$. Note that the last player in team $m$ is player $m d$, so the first player after team $m$ is player $m d + 1$ (mod $p$). There are $t - 1 \geq d$ other teams, and each team has $d$ players, where $d \geq 3$. We now consider two separate cases. \\ 

Case 1: Let $c = 1$. There are $d^2 \geq 3d$ consecutive players from teams other than team $m$. After all the members of team $m$'s first move, the next $d$ teams can all do the following.
\begin{equation}
\begin{split}
\text {Players from $m d + 1$ to $(m + 1)d$ do:} & \ (2)S_{1} \rightarrow S_{2}.\\
\text {Players from $(m + 1)d + 1$ to $(m + 2)d$ do:} & \ (2)S_{1} \rightarrow S_{2}.\\
\text {Players from $(m + 2)d + 1$ to $(m + 3)d$ do:} & \
\begin{cases}
(2)S_{2} \rightarrow S_{3} \wedge S_{1} & \text {if k = 1}\\
(2)S_{2} \rightarrow S_{3} & \text {if k > 1}.
\end{cases}
\end{split}
\end{equation}
Since all of these $3d$ players are not from team $m$, this contradicts Property \ref{property 2}.\\

Case 2: Let $c > 1$. The stealing strategy described in Case 2 of the proof of Property \ref{property 2} takes $c + 1$ moves for each repetition, and $(c + 1)d$ moves total. There are $d^2 \geq (c + 1)d$ consecutive players from other teams, so after all the members of team $m$ complete their first move, the next $d$ teams can do the sequence of $(c + 1)d$ moves. This contradicts Property \ref{property 2}, so team $m$ does not have a winning strategy.
\end{proof}

\begin{lem}\label{DI7}
For any $t = c + 2$ and $d=2$, no team has a winning strategy when $n \geq 6c^2 + 18c +12$ and $t \geq 2c$.
\end{lem}

\begin{proof}
Suppose team $m$ has a winning strategy $(1 \leq m \leq t)$. For any $c$, after all the members of team $m$'s first move, the remaining players can do the following.

\begin{equation}
\begin{split}
\text {Step 1:} & \ (c+1)S_{1} \rightarrow S_{2}. \\
\vdots \\
\text {Step 2c + 2:} & \ (c+1)S_{1} \rightarrow S_{2}.\\
\text {Step 2c + 3:} & \ \text{Anything since this player is on team $m$}.\\
\text {Step 2c + 4:} & \ \text{Anything since this player is on team $m$}.\\
\text {Step 2c + 5:} & \ (c+1)S_{1} \rightarrow S_{2}.\\
\vdots\\
\text {Step 4c + 6:} & \ (c+1)S_{1} \rightarrow S_{2}.\\
\text {Step 4c + 7:} & \ \text{Anything since this player is on team $m$.} \\
\text {Step 4c + 8:} & \ \text{Anything since this player is on team $m$.} \\
\end{split}
\end{equation}

We now consider the following two cases.\\

Case 1: If steps $2c + 3$ and $2c + 4$ are both $(c + 1)S_{2} \rightarrow S_{3} \wedge S_{1}$ when $k = 1$ or $(c + 1)S_{2} \rightarrow S_{3}$ when $k > 1$, then steps $1$ through $2c + 4$ form the sequence described in Property $2$. The same is true for steps $4c + 7$ and $4c + 8$ which would make the sequence in steps $2c + 5$ through $4c + 8$.\\

Case 2: If one of the steps from $2c +3$ and $2c +4$ is not $(c + 1)S_{2} \rightarrow S_{3} \wedge S_{1}$ if $k = 1$ or $(c + 1)S_{2} \rightarrow S_{3}$ if $k > 1$, and one of steps $4c + 7$ or $4c + 8$ is not either, then the players not on team $m$ can do the following after step $4c + 8$ described above.

\begin{equation}
\begin{split}
\text {Step 4c + 9:} & \ (c+1)S_{1} \rightarrow S_{2}. \\
\text {Step 4c + 10:} & \ (c+1)S_{1} \rightarrow S_{2}. \\
\text {Step 4c + 11:} & \ 
\begin{cases}
(c + 1)S_{2} \rightarrow S_{3} \wedge S_{1} & \text {if k = 1}\\
(c + 1)S_{2} \rightarrow S_{3} & \text {if k > 1}.
\end{cases}\\
\text {Step 4c + 12:} & \ 
\begin{cases}
(c + 1)S_{2} \rightarrow S_{3} \wedge S_{1} & \text {if k = 1}\\
(c + 1)S_{2} \rightarrow S_{3} & \text {if k > 1}.
\end{cases}\\
\end{split}
\end{equation}
Note that steps $2c + 3$ and $2c + 4$ take away at most three of $S_{2}$, and steps $4c + 7$ and $4c + 8$ also take away at most three of $S_{2}$. Also, note that steps $1$ through $2c + 3$ and $2c + 5$ through $4c + 6$ generate eight of $S_{2}$ in total, so there will be at least two $S_{2}$'s remaining. Therefore for steps $4c + 9$ and $4c + 10$,instead of doing $(c + 1)S_{1} \rightarrow S_{2}$ the players can instead do $c S_{1} \wedge c S_{2} \rightarrow S_{3}$ if $k = 1$, or $(c + 1)S_{1} \wedge c S_{2} \rightarrow S_{3}$ if $k > 1$. Since team $m$ has the winning strategy, suppose Player $a$ on team $m$ would have made the last move. However, after this sequence of moves Player $a - 2$ on team $m - 1$ will make the last move, which is a contradiction. This proves Lemma \ref{DI7}.
\end{proof}

By Lemmas \ref{DI6} and \ref{DI7}, Theorem \ref{DI5} is proved.
\end{proof}

\textbf{Theorem \ref{DI8}.}
\textit{
Let $p \geq 6$ and $c = 1$. If there are two teams, one with $p - 2$ players  and the other with two players, then the larger team has a winning strategy for $n \geq 36$ in the Team Generalized Zeckendorf Game. 
}

\begin{proof}
We prove this result independently in several different cases based on the arrangement of the players.\\

Case 1: Suppose $p = 6$, and the $4$-player team consists of $4$ consecutive players, then the $2$-player alliance will have $2$ consecutive players. If we consider the $4$-player team as two teams each with two consecutive players, then by Lemma \ref{DI7} the $2$-player alliance does not have a winning strategy. Since there are only two teams this is analogous to the regular $2$-player game where one team must have a winning strategy, so the $4$-player team must have a winning strategy.\\

Case 2: Suppose again that $p = 6$, but the $4$-player team is separated in two parts, each with two consecutive players. Note that the sequence of all six players is equivalent to two rounds of a $3$-player game in which two of the players are on the same team. According to Lemma \ref{DI4} the single player does not have a winning strategy in this case, so once again the $4$-player team has a winning strategy.\\

Case 3: Suppose the larger team is separated in two parts, where one part has at least three consecutive players. Note that this will always be the case for $p \geq 7$ by the pigeonhole principle. Suppose for contradiction that the smaller team has a winning strategy, then there exists a Player $q$ from the smaller team who makes the last move. Let the three consecutive players on the larger team be $a$, $a + 1$, and $a + 2$. They can do the following moves to steal the winning strategy.
\begin{equation}
\begin{split}
\text {Player a:} & \ (2)S_{1} \rightarrow S_{2}.\\
\text {Player a + 1:} & \ (2)S_{1} \rightarrow S_{2}. \\
\text {Player a + 2:} & \ 
\begin{cases}
(2)S_{2} \rightarrow S_{3} \wedge S_{1} & \text {if k = 1}\\
(2)S_{2} \rightarrow S_{3} & \text {if k > 1}.
\end{cases}\\
\end{split}
\end{equation}
Note that if Player $a + 1$ instead does $S_{1} \wedge S_{2} \rightarrow S_{3}$ if $k = 1$, or $2 S_{1} \wedge S_{2} \rightarrow S_{3}$ if $k > 1$, the Player $q - 1$ will now make the last move. Since the players on the smaller team are separated, Player $q - 1$ belongs to the larger team, which is a contradiction. Thus, the larger team has a winning strategy.\\

Case 4: Suppose $p = 7$ and the players on the larger team are all consecutive, and for contradiction assume the smaller team has a winning strategy. Let the first of the five consecutive players be Player $a$, then the larger team can do the following.
\begin{equation}
\begin{split}
\text {Step 1: Player a:} & \ (2)S_{1} \rightarrow S_{2}. \\
\vdots \\
\text {Step 5: Player a + 4:} & \ (2)S_{1} \rightarrow S_{2}. \\
\text {Step 6: Player a + 5:} & \ \text {Anything since this player is on the smaller team.} \\
\text {Step 7: Player a + 6:} & \ \text {Anything since this player is on the smaller team.} \\
\text {Step 8: Player a:} & \ (2)S_{1} \rightarrow S_{2}.\\
\text {Step 9: Player a + 1:} & \ (2)S_{1} \rightarrow S_{2}.\\
\text {Step 10: Player a + 2:} & \ (2)S_{1} \rightarrow S_{2}.\\
\text {Step 11: Player a + 3:} & \ 
\begin{cases}
(2)S_{2} \rightarrow S_{3} \wedge S_{1} & \text {if k = 1}\\
(2)S_{2} \rightarrow S_{3} & \text {if k > 1}.
\end{cases}\\
\text {Step 12: Player a + 4:} & \ 
\begin{cases}
(2)S_{2} \rightarrow S_{3} \wedge S_{1} & \text {if k = 1}\\
(2)S_{2} \rightarrow S_{3} & \text {if k > 1}.
\end{cases}\\
\end{split}
\end{equation}
Note that steps $6$ and $7$ can take away at most four of $S_{2}$ in total, and steps $1$, $2$, $3$, $4$, $5$, and $8$ generate six of $S_{2}$ in total. Therefore, after step $8$ there will be at least two of $S_{2}$ remaining. Let $q$ be the player from the smaller team who makes the final move. Then Players $a + 1$ and $a + 2$ in steps $9$ and $10$ can instead do $S_{1} \wedge S_{2} \rightarrow S_{3}$ if $k = 1$, or $2 S_{1} \wedge S_{2} \rightarrow S_{3}$ if $k > 1$. Then Player $q - 2$ on the larger team becomes the player to take the last move, which is a contradiction so the larger team has a winning strategy.\\

Case 5: Suppose $p \geq 8$ and the players on the larger team are all consecutive, and for contradiction assume the smaller team has a winning strategy, with Player $q$ on the smaller team making the final move. Let the first of the consecutive players be Player $a$, then the larger team can do the following.
\begin{equation}
\begin{split}
\text {Step 1: Player a:} & \ (2)S_{1} \rightarrow S_{2}. \\
\vdots \\
\text {Step 4: Player a + 3:} & \ (2)S_{1} \rightarrow S_{2}.\\
\text {Step 5: Player a + 4:} & \ 
\begin{cases}
(2)S_{2} \rightarrow S_{3} \wedge S_{1} & \text {if k = 1}\\
(2)S_{2} \rightarrow S_{3} & \text {if k > 1}.
\end{cases}\\
\text {Step 6: Player a + 5:} & \ 
\begin{cases}
(2)S_{2} \rightarrow S_{3} \wedge S_{1} & \text {if k = 1}\\
(2)S_{2} \rightarrow S_{3} & \text {if k > 1}.
\end{cases}\\
\end{split}
\end{equation}
Note that Players $a$ through $a + 5$ are all on the large alliance, and that Players $a + 2$ and $a + 3$ can instead do $S_{1} \wedge S_{2} \rightarrow S_{3}$ if $k = 1$, or $2 S_{1} \wedge S_{2} \rightarrow S_{3}$ if $k > 1$. Then Player $q - 2$ on the larger team becomes the player to take the last move, which is a contradiction so the larger team has a winning strategy.

These five cases prove Theorem \ref{DI8}.
\end{proof}

\section{Future Work}
Although we have proved some significant results on winning strategies of Generalized (c,k)-nacci game, there are still many interesting questions worth exploring. First of all, we have not proved which player has a winning strategy for all types of Generalized Zeckendorf Games, but rather only for some of them. It would be very interesting if we could extend our results and techniques for finding winning strategies to other generalized sequences described by linear recurrence relations. For instance, if we extend the game to work on more generalized sequences, would the stealing strategy technique still apply? We could also further discuss which player has a winning strategy in the two-player version of the game, and whether a player or an alliance has a winning strategy in the multiplayer or multialliance games of more generalized linear recurrence relations.

Additionally, we could try to improve our lower bound for the two-player Generalized $(c,k)$-nacci Zeckendorf Game, and prove which player has a winning strategy when $n$ is smaller than $(c+1)^3+c+1$. We have found improved bounds for specific cases such as the Zeckendorf Game on the Fibonacci numbers and the Tribonacci Game, so we would like to find a general pattern for the lower bound where Theorem \ref{thm3} holds. Similarly, we can improve our lower bounds on $n$ related to the multiplayer and multialliance Generalized $(c,k)$-nacci Zeckendorf Game. 
\newpage


\appendix

\section{Proof of Lemma \ref{(c,1)} When d is Even}\label{appendix 1}
Now we consider the case when $d$ is even, and we suppose for contradiction that Player $2$ has a winning strategy. Note that from row $1$ to row $c$, each node only has one child, so Player $2$ has a winning strategy for all these $c$ rows. Row $c$ is Player $1$'s turn since $c$ is odd, so both nodes in row $c+1$ have a winning strategy for Player $2$. We can follow this pattern following the same moves as when $c$ is even until row $c^2$. The only slight difference is on turns $kc$ where k is even, since turn $kc$ is Player $2$'s turn and Player $2$ also has the winning strategy. In this case at least one of the children in row $kc + 1$ must have a winning strategy for Player $2$. Since both nodes in row $kc + 1$ have children back on the main part of the tree, the node in row $kc+2$ must have a winning strategy for Player $2$. The node $\{1^{(n-c^3+2c-1)} \wedge c^{(c-1)} \wedge (cc-1)^{(c-1)}\}$ in row $c^2$ is Player $1$'s turn, so all of its children in row $c^2 + 1$ also have a winning strategy for Player $2$. Since $\{1^{(n-c^3+2c-1)} \wedge ccc-2c+1\}$ in row $c^2+1$ has only one child, $\{1^{(n-c^3+2c-1)} \wedge c \wedge ccc-2c+1\}$ in row $c^2+2$ must also have a winning strategy for Player $2$. Then Player $1$ must have a winning strategy for the equivalent node in row $C^2+3$. However, since Player $2$ has a winning strategy for $\{1^{(n-c^3-c)} \wedge (cc-1)^{(c)}\}$ in row $c^2+1$, at least one of its children must also have a winning strategy for Player $2$. Both children are parents to the node $\{1^{(n-c^3+c-1)} \wedge c \wedge ccc-2c+1\}$ in row $c^2+3$, and since row $c^2+2$ is Player $1$'s turn, $\{1^{(n-c^3+c-1)} \wedge c \wedge ccc-2c+1\}$ in row $c^2+3$ must also have a winning strategy for Player $2$. This is a contradiction, so Player $2$ must have a winning strategy.

\section{Proof of Lemma \ref{(c,k)}}\label{appendix 2}
Let $d=c-1$. Note that Lemma \ref{(c,k)} is equivalent to the following statement: for $(d,k)$-nacci Game, for any $c\geq 2$, $k\geq3$, and $n\geq c^3+c$, when $d$ is odd, Player $2$ always has a winning strategy; when $d$ is even, Player $1$ always has a winning strategy.

First, we prove that for the $(d,k)$-nacci Game, for any $c\geq 2$, $k\geq3$, and $n\geq c^3+c$, when $d$ is odd, Player $2$ always has a winning strategy.

Suppose for contradiction that Player $1$ has a winning strategy. Note that from row $1$ to row $c$, each node only has $1$ child, so Player $1$ has a winning strategy for all these $c$ rows. In row $c$ Player $1$ has a winning strategy for $\{1^{n-(c-1)c}\wedge c^{c-1}\}$, and since it is Player $2$'s turn, Player $1$ has a winning strategy for all of its children, including $\{1^{n-c^2}\wedge c^c\}$ and $\{1^{n-c^2}\wedge cc\}$ in row $c+1$. Since $\{1^{n-c^2}\wedge cc\}$ in row $c+1$ has only one child, for its only child $\{1^{n-c^2-c}\wedge c\wedge cc\}$ in row $c+2$, Player $1$ has a winning strategy. We call the process from row $1$ to row $c$ as round $1$, which contains $c$ rows in total.

Starting from the node $\{1^{n-c^2}\wedge cc\}$ in row $c+1$, we repeat the same procedure as in round $1$. We repeat this procedure $c$ times until row $c^2$.

As Player $1$ has a winning strategy for $\{1^{n-c(c-1)-c^2(c-1)}\wedge c^{c-1}\wedge cc^{c-1}\}$ in row $c^2$, and it is Player $2$'s turn in row $c^2$, then Player $1$ has a winning strategy for all its children, which includes the $3$ nodes in row $c^2+1$ as shown in the diagram. Note that $\{1^{n-c^3}\wedge ccc\}$ in row $c^2+1$ has one child, so Player $1$ has a winning strategy for $\{1^{n-c^3-c}\wedge c \wedge ccc\}$ in row $c^2+2$. Since it is equivalent to $\{1^{n-c^3-c}\wedge c \wedge ccc\}$ in row $c^2+2$, Player $2$ has a winning strategy for $\{1^{n-c^3-c}\wedge c\wedge ccc\}$ in row $c^2+3$. 

On the other hand, since Player $1$ has a winning strategy for $\{1^{n-c^3}\wedge cc^c\}$ in row $c^2+1$, Player $1$ has a winning strategy for at least one of its children, which is either $\{1^{n-c^3-c}\wedge c\wedge cc^c\}$ or $\{1^{n-c^3}\wedge ccc\}$ in row $c^2+2$. Since $\{1^{n-c^3-c}\wedge c\wedge ccc\}$ in row $c^2+3$ is a child of both nodes, and it is Player $2$'s turn in row $c^2+2$, it follows that Player $1$ has a winning strategy for $\{1^{n-c^3-c}\wedge c\wedge ccc\}$ in row $c^2+3$. This is a contradiction, so we have proved that Player $2$ has a winning strategy when $d$ is odd.

Next, we prove that for the $(d,k)$-nacci Game, for any $c\geq 2$, $k\geq3$, and $n\geq c^3+c$, when $d$ is even, Player $1$ always has a winning strategy.

We suppose for contradiction that Player $2$ has a winning strategy. Note that from row $1$ to row $c$, each node has only $1$ child, so Player $2$ has a winning strategy for all these $c$ rows. In row $c$ Player $2$ has a winning strategy for $\{1^{n-(c-1)c}\wedge c^{c-1}\}$, and since it is Player $1$'s turn, Player $2$ has a winning strategy for all its children, including $\{1^{n-c^2}\wedge c^c\}$ and $\{1^{n-c^2}\wedge cc\}$ in row $c+1$. Since $\{1^{n-c^2}\wedge cc\}$ has only $1$ child, we can follow a similar process to that of the first $c$ rows from $\{1^{n-c^2}\wedge cc\}$ in row $c+1$ until $\{1^{n-c^2-(c-1)c}\wedge c^{c-1}\wedge cc\}$ in row $2c$. Since each node has only one child, Player $2$ has a winning strategy for all these $c$ nodes.

Since Player $2$ has a winning strategy for $\{1^{n-c^2-(c-1)c}\wedge c^{c-1}\wedge cc\}$ in row $2c$, Player $2$ has a winning strategy for at least one of its children, which is either $\{1^{n-2c^2}\wedge c^c \wedge cc\}$ or $\{1^{n-2c^2}\wedge cc^2\}$ in row $2c+1$. Note that $\{1^{n-2c^2-c}\wedge c\wedge cc^2\}$ in row $2c+2$ is a child of both nodes, and it is Player $1$'s turn in row $2c+1$, so Player $2$ has a winning strategy for $\{1^{n-2c^2-c}\wedge c\wedge cc^2\}$ in row $2c+2$.

We call the process from row $1$ to row $c$ as round $1$, the process from row $c+1$ to row $2c$ as round $2$, and so on, where each round contains $c$ consecutive rows. We can repeat the same process for $c$ rounds until row $c^2$. Also, note that from row $2tc$ to row $2tc+2$ (where $t$ is a positive integer and $1\leq t \leq d/2$), the proof is the same as the proof from row $2c$ to row $2c+2$.

As Player $2$ has a winning strategy for $\{1^{n-c(c-1)-c^2(c-1)}\wedge c^{c-1} \wedge cc^{c-1}\}$ in row $c^2$, and it is Player $1$'s turn in row $c^2$, then Player $2$ has a winning strategy for all its children, which includes the $3$ nodes in row $c^2+1$ as shown in the diagram. Note that $\{1^{n-c^3}\wedge ccc\}$ in row $c^2+1$ has one child, so Player $2$ has a winning strategy for $\{1^{n-c^3-c}\wedge c\wedge ccc\}$ in row $c^2+2$. Since it is equivalent to $\{1^{n-c^3-c}\wedge c\wedge ccc\}$ in row $c^2+2$, Player $1$ has a winning strategy for $\{1^{n-c^3-c}\wedge c\wedge ccc\}$ in row $c^2+3$.

On the other hand, since Player $2$ has a winning strategy for $\{1^{n-c^3}\wedge cc^c\}$ in row $c^2+1$, Player $2$ has a winning strategy for at least one of its children, which is either $\{1^{n-c^3-c}\wedge c\wedge cc^c\}$ or $\{1^{n-c^3}\wedge ccc\}$ in row $c^2+2$. Since $\{1^{n-c^3-c}\wedge c\wedge ccc\}$ in row $c^2+3$ is a child of both nodes, and it is Player $1$'s turn in row $c^2+2$, it follows that Player $2$ has a winning strategy for $\{1^{n-c^3-c}\wedge c\wedge ccc\}$ in row $c^2+3$. This is a contradiction, so we have proved that Player $1$ has a winning strategy when $d$ is even, which completes the proof.


\newpage
\section{Tree for Lemma \ref{(c,1)} When d is Even}\label{appendix 3}
\begin{figure}[h]
    \centering
\resizebox{12cm}{!}{
\begin{tikzcd}[ampersand replacement=\&]
	player \& turn \\
	1 \& 1 \&\& \textcolor{rgb,255:red,92;green,92;blue,214}{1^n} \\
	2 \& 2 \&\& \textcolor{rgb,255:red,92;green,92;blue,214}{1^{(n-c)} \wedge c} \&\& {} \\
	2 \& c \&\& \textcolor{rgb,255:red,92;green,92;blue,214}{\ 1^{(n-c^2+c)} \wedge c^{(c-1)}} \&\&\& {} \\
	1 \& {c+1} \& \textcolor{rgb,255:red,92;green,92;blue,214}{1^{(n-c^2)} \wedge c^{(c)}} \&\& \textcolor{rgb,255:red,92;green,92;blue,214}{1^{(n-c^2+1)} \wedge cc-1} \&\&\\
	2 \& {c+2} \&\& \textcolor{rgb,255:red,92;green,92;blue,214}{1^{(n-c^2-c+1)} \wedge c \wedge cc-1} \\
	2 \& 2c \&\& \textcolor{rgb,255:red,92;green,92;blue,214}{1^{(n-2c^2+c+1)} \wedge c^{(c-1)} \wedge cc-1} \\
	1 \& {2c+1} \& \textcolor{rgb,255:red,92;green,92;blue,214}{1^{(n-2c^2+1)} \wedge c^{(c)} \wedge cc-1} \& \textcolor{rgb,255:red,92;green,92;blue,214}{1^{(n-2c^2+2)} \wedge (cc-1)^{(2)}} \&\&\\
	2 \& {2c+2} \&\& \textcolor{rgb,255:red,92;green,92;blue,214}{1^{(n-2c^2-c+2)} \wedge c \wedge (cc-1)^{(2)}} \\
	2 \& {c^2} \&\& \textcolor{rgb,255:red,92;green,92;blue,214}{1^{(n-c^3+2c-1)} \wedge c^{(c-1)} \wedge (cc-1)^{(c-1)}} \\
	1 \& {c^2+1} \& \textcolor{rgb,255:red,92;green,92;blue,214}{1^{(n-c^3+c-1)} \wedge c^{(c)} \wedge (cc-1)^{(c-1)}} \& \textcolor{rgb,255:red,92;green,92;blue,214}{1^{(n-c^3-c)} \wedge (cc-1)^{(c)}} \& \textcolor{rgb,255:red,92;green,92;blue,214}{1^{(n-c^3+2c-1)} \wedge ccc-2c+1} \\
	2 \& {c^2+2} \& \textcolor{rgb,255:red,92;green,92;blue,214}{1^{(n-c^3)} \wedge c \wedge (cc-1)^{(c)}} \& \textcolor{rgb,255:red,92;green,92;blue,214}{1^{(n-c^3+2c-1)} \wedge ccc-2c+1} \& \textcolor{rgb,255:red,92;green,92;blue,214}{1^{(n-c^3+c-1)} \wedge c \wedge ccc-2c+1} \\
	1 \& {c^2+3} \&\& \textcolor{rgb,255:red,92;green,214;blue,92}{1^{(n-c^3+c-1)} \wedge c \wedge ccc-2c+1} \& {}
	\arrow[no head, from=2-4, to=3-4]
	\arrow[dashed, no head, from=3-4, to=4-4]
	\arrow[no head, from=4-4, to=5-3]
	\arrow[no head, from=4-4, to=5-5]
	\arrow[no head, from=5-3, to=6-4]
	\arrow[no head, from=5-5, to=6-4]
	\arrow[dashed, no head, from=6-4, to=7-4]
	\arrow[no head, from=7-4, to=8-3]
	\arrow[no head, from=7-4, to=8-4]
	\arrow[no head, from=10-4, to=11-3]
	\arrow[no head, from=8-4, to=9-4]
	\arrow[no head, from=8-3, to=9-4]
	\arrow[dashed, no head, from=9-4, to=10-4]
	\arrow[no head, from=10-4, to=11-4]
	\arrow[no head, from=10-4, to=11-5]
	\arrow[no head, from=11-3, to=12-3]
	\arrow[no head, from=11-4, to=12-3]
	\arrow[no head, from=11-4, to=12-4]
	\arrow[no head, from=11-5, to=12-5]
	\arrow[no head, from=12-3, to=13-4]
	\arrow[no head, from=12-4, to=13-4]
	\arrow[dashed, no head, from=3-1, to=4-1]
	\arrow[dashed, no head, from=3-2, to=4-2]
	\arrow[dashed, no head, from=6-2, to=7-2]
	\arrow[dashed, no head, from=6-1, to=7-1]
	\arrow[dashed, no head, from=9-2, to=10-2]
	\arrow[dashed, no head, from=9-1, to=10-1]
\end{tikzcd}
}
    \caption{$(d, 1)$-nacci tree depicting the proof of Theorem \ref{(c,1)}, where $d=c-1$ is odd. Red indicates a winning strategy for Player $1$, and blue indicates a winning strategy for Player $2$. Green indicates a winning strategy for both players, which is a contradiction.}
    \label{fig:the (c,1)-Tree}
\end{figure}


\newpage
\section{Tree for Lemma \ref{(c,2)} When d is Even}\label{appendix 4}
\begin{figure}[h]
    \centering
\resizebox{12cm}{!}{
\begin{tikzcd}[ampersand replacement=\&]
	player \& turn \\
	1 \& 1 \&\& \textcolor{rgb,255:red,92;green,92;blue,214}{1^n} \\
	2 \& 2 \&\& \textcolor{rgb,255:red,92;green,92;blue,214}{1^{(n-c)} \wedge c} \\
	2 \& c \&\& \textcolor{rgb,255:red,92;green,92;blue,214}{1^{(n-c^2+c)} \wedge c^{(c-1)}} \\
	1 \& {c+1} \& \textcolor{rgb,255:red,92;green,92;blue,214}{1^{(n-c^2)} \wedge c^{(c)}} \&\& \textcolor{rgb,255:red,92;green,92;blue,214}{1^{(n-c^2)} \wedge cc} \\
	2 \& {c+2} \&\& \textcolor{rgb,255:red,92;green,92;blue,214}{1^{(n-c^2-c)} \wedge c \wedge cc} \\
	2 \& 2c \&\& \textcolor{rgb,255:red,92;green,92;blue,214}{1^{(n-2c^2+c)} \wedge c^{(c-1)} \wedge cc} \\
	1 \& {2c+1} \& \textcolor{rgb,255:red,92;green,92;blue,214}{1^{(n-2c^2)} \wedge c^{(c)} \wedge cc} \& \textcolor{rgb,255:red,92;green,92;blue,214}{1^{(n-2c^2)} \wedge cc^{(2)}} \\
	2 \& {2c+2} \&\& \textcolor{rgb,255:red,92;green,92;blue,214}{1^{(n-2c^2-c)} \wedge c \wedge cc^{(2)}} \\
	2 \& {c^2} \&\& \textcolor{rgb,255:red,92;green,92;blue,214}{1^{(n-c^3+c)} \wedge c^{(c-1)} \wedge cc^{(c-1)}} \\
	1 \& {c^2+1} \& \textcolor{rgb,255:red,92;green,92;blue,214}{1^{(n-c^3)} \wedge c^{(c)} \wedge cc^{(c-1)}} \& \textcolor{rgb,255:red,92;green,92;blue,214}{1^{(n-c^3)} \wedge cc^{(c)}} \& \textcolor{rgb,255:red,92;green,92;blue,214}{1^{(n-c^3+1)} \wedge ccc-1} \\
	2 \& {c^2+2} \& \textcolor{rgb,255:red,92;green,92;blue,214}{1^{(n-c^3-c)} \wedge c \wedge cc^{(c)}} \& \textcolor{rgb,255:red,92;green,92;blue,214}{1^{(n-c^3+1)} \wedge ccc-1} \& \textcolor{rgb,255:red,92;green,92;blue,214}{1^{(n-c^3-c+1)} \wedge c \wedge ccc-1} \\
	1 \& {c^2+3} \&\& \textcolor{rgb,255:red,92;green,214;blue,92}{1^{(n-c^3-c+1)} \wedge c \wedge ccc-1} \\
	\&\&\&\& {}
	\arrow[no head, from=2-4, to=3-4]
	\arrow[dashed, no head, from=3-4, to=4-4]
	\arrow[no head, from=4-4, to=5-3]
	\arrow[no head, from=4-4, to=5-5]
	\arrow[no head, from=5-3, to=6-4]
	\arrow[no head, from=5-5, to=6-4]
	\arrow[dashed, no head, from=6-4, to=7-4]
	\arrow[no head, from=7-4, to=8-4]
	\arrow[no head, from=7-4, to=8-3]
	\arrow[no head, from=8-4, to=9-4]
	\arrow[no head, from=8-3, to=9-4]
	\arrow[dashed, no head, from=9-4, to=10-4]
	\arrow[no head, from=10-4, to=11-4]
	\arrow[no head, from=10-4, to=11-3]
	\arrow[no head, from=10-4, to=11-5]
	\arrow[no head, from=11-3, to=12-3]
	\arrow[no head, from=11-4, to=12-3]
	\arrow[no head, from=11-5, to=12-5]
	\arrow[no head, from=12-3, to=13-4]
	\arrow[no head, from=11-4, to=12-4]
	\arrow[no head, from=12-4, to=13-4]
	\arrow[dashed, no head, from=3-2, to=4-2]
	\arrow[dashed, no head, from=3-1, to=4-1]
	\arrow[dashed, no head, from=6-2, to=7-2]
	\arrow[dashed, no head, from=6-1, to=7-1]
	\arrow[dashed, no head, from=9-2, to=10-2]
	\arrow[dashed, no head, from=9-1, to=10-1]
\end{tikzcd}
}
    \caption{$(d, 2)$-nacci tree depicting the proof of Theorem \ref{(c,2)}, where $d=c-1$ is even. Red indicates a winning strategy for Player $1$, and blue indicates a winning strategy for Player $2$. Green indicates a winning strategy for both players, which is a contradiction.}
    \label{fig:the (c,2)-Tree}
\end{figure}

\newpage
\section{Tree for Lemma \ref{(c,k)} When d is Even}\label{appendix 5}
\begin{figure}[h]
    \centering
\resizebox{11cm}{!}{
\begin{tikzcd}[ampersand replacement=\&]
	player \& turn \\
	1 \& 1 \&\& \textcolor{rgb,255:red,92;green,92;blue,214}{1^n} \\
	2 \& 2 \&\& \textcolor{rgb,255:red,92;green,92;blue,214}{1^{(n-c)}} \\
	2 \& c \&\& \textcolor{rgb,255:red,92;green,92;blue,214}{1^{(n-c^2+c)} \wedge c^{(c-1)}} \\
	1 \& {c+1} \& \textcolor{rgb,255:red,92;green,92;blue,214}{1^{(n-c^2)} \wedge c^{(c)}} \&\& \textcolor{rgb,255:red,92;green,92;blue,214}{1^{(n-c^2)} \wedge cc} \\
	2 \& {c+2} \&\& \textcolor{rgb,255:red,92;green,92;blue,214}{1^{(n-c^2-c)} \wedge c \wedge cc} \\
	2 \& 2c \&\& \textcolor{rgb,255:red,92;green,92;blue,214}{1^{(n-2c^2+c)} \wedge c^{(c-1)} \wedge cc} \\
	1 \& {2c+1} \& \textcolor{rgb,255:red,92;green,92;blue,214}{1^{(n-2c^2)} \wedge c^{(c)} \wedge cc} \& \textcolor{rgb,255:red,92;green,92;blue,214}{1^{(n-2c^2)} \wedge cc^{(2)}} \\
	2 \& {2c+2} \&\& \textcolor{rgb,255:red,92;green,92;blue,214}{1^{(n-2c^2-c)} \wedge c \wedge cc^{(2)}} \\
	2 \& {c^2} \&\& \textcolor{rgb,255:red,92;green,92;blue,214}{1^{(n-c^3+c)} \wedge c^{(c-1)} \wedge cc^{(c-1)}} \\
	1 \& {c^2+1} \& \textcolor{rgb,255:red,92;green,92;blue,214}{1^{(n-c^3)} \wedge c^{(c)} \wedge cc^{(c-1)}} \& \textcolor{rgb,255:red,92;green,92;blue,214}{1^{(n-c^3)} \wedge cc^{(c)}} \& \textcolor{rgb,255:red,92;green,92;blue,214}{1^{(n-c^3)} \wedge ccc} \\
	2 \& {c^2+2} \& \textcolor{rgb,255:red,92;green,92;blue,214}{1^{(n-c^3-c)} \wedge c \wedge cc^{(c)}} \& \textcolor{rgb,255:red,92;green,92;blue,214}{1^{(n-c^3)} \wedge ccc} \& \textcolor{rgb,255:red,92;green,92;blue,214}{1^{(n-c^3-c)} \wedge c \wedge ccc} \\
	1 \& {c^2+3} \&\& \textcolor{rgb,255:red,92;green,214;blue,92}{1^{(n-c^3-c)} \wedge c \wedge ccc}
	\arrow[no head, from=2-4, to=3-4]
	\arrow[dashed, no head, from=3-4, to=4-4]
	\arrow[no head, from=4-4, to=5-5]
	\arrow[no head, from=4-4, to=5-3]
	\arrow[no head, from=5-5, to=6-4]
	\arrow[no head, from=5-3, to=6-4]
	\arrow[dashed, no head, from=6-4, to=7-4]
	\arrow[no head, from=7-4, to=8-4]
	\arrow[no head, from=7-4, to=8-3]
	\arrow[no head, from=8-4, to=9-4]
	\arrow[no head, from=8-3, to=9-4]
	\arrow[dashed, no head, from=9-4, to=10-4]
	\arrow[no head, from=10-4, to=11-3]
	\arrow[no head, from=10-4, to=11-5]
	\arrow[no head, from=10-4, to=11-4]
	\arrow[no head, from=11-3, to=12-3]
	\arrow[no head, from=11-4, to=12-3]
	\arrow[no head, from=12-4, to=11-4]
	\arrow[no head, from=11-5, to=12-5]
	\arrow[no head, from=12-4, to=13-4]
	\arrow[dashed, no head, from=3-2, to=4-2]
	\arrow[dashed, no head, from=3-1, to=4-1]
	\arrow[dashed, no head, from=6-2, to=7-2]
	\arrow[dashed, no head, from=6-1, to=7-1]
	\arrow[dashed, no head, from=9-2, to=10-2]
	\arrow[dashed, no head, from=9-1, to=10-1]
	\arrow[no head, from=12-3, to=13-4]
\end{tikzcd}
}
    \caption{$(d, k)$-nacci tree depicting the proof of Theorem \ref{(c,k)}, where $d=c-1$ is even and $k\geq3$. Red indicates a winning strategy for Player $1$, and blue indicates a winning strategy for Player $2$. Green indicates a winning strategy for both players, which is a contradiction.}
    \label{fig:the (c,k)-Tree}
\end{figure}

\newpage


\medskip

\noindent MSC2020: 91A05, 91A06

\end{document}